\documentclass[final,leqno,oneeqnum,onethmnum]{siamltex704}
\usepackage{graphicx,epsfig}
\usepackage{amssymb,color}
\usepackage{amsmath}
\usepackage{exscale}
\def\gph {\mathop{\rm gph\,  }}

\newtheorem{theo}{Theorem}
\newtheorem{lem}[theo]{Lemma}
\newtheorem{prop}[theo]{Proposition}
\newtheorem{cor}[theo]{Corollary}
\newtheorem{rem}{Remark}
\newtheorem{exa}{Example}
\newtheorem{definiti}{Definition}

\usepackage[normalem]{ulem}
\usepackage[nocompress]{cite}

\title{Calmness modulus of linear semi-infinite programs \thanks{%
This research has been partially supported by Grants MTM2011-29064-C03
(02-03) from MINECO, Spain ACOMP/2013/062 from Generalitat Valenciana, Spain,   Grant  C10E08  from  ECOS-SUD, and Grant DP110102011 from the Australian
Research Council. }}
\author{M.J. C\'{a}novas\thanks{%
Operations Research Center, Miguel Hern\'{a}ndez University of Elche, 03202
Elche (Alicante), Spain (canovas@umh.es).}
\and
A.Y. Kruger\thanks{%
Centre for Informatics and Applied Optimization, School of Science,
Information Technology and Engineering, University of Ballarat, POB 663,
Ballarat, Vic, 3350, Australia (a.kruger@ballarat.edu.au).}\
\and M.A. L\'{o}pez\thanks{%
Department of Statistics and Operations Research, University of Alicante, Spain and Adjunct professor, University of Ballarat
03071 Alicante, Spain (marco.antonio@ua.es).}
\and J. Parra\thanks{Operations Research Center, Miguel Hern\'{a}ndez University of Elche, 03202
Elche (Alicante), Spain  (parra@umh.es).}
 \and M.A. Th\'{e}ra\thanks{%
Laboratoire XLIM, UMR-CNRS 6172, University of Limoges, France and Adjunct professor, University of Ballarat
(michel.thera@unilim.fr). }}

\begin{document}
%siam_id=76781
%CODEN=SJOPE8
%\slugger{siopt}{2010}{20}{4}{2080--2096}
\maketitle

\begin{abstract}
Our main goal is to compute or estimate the calmness modulus of the argmin
mapping of linear semi-infinite optimization problems under canonical
perturbations, i. \hspace{-0.15cm}e., perturbations of the objective
function together with continuous perturbations of the right-hand-side of
the constraint system (with respect to an index ranging in a compact
Hausdorff space). Specifically, we provide a lower bound on the calmness
modulus for semi-infinite programs with unique optimal solution which turns
out to be the exact modulus when the problem is finitely constrained. The
relationship between the calmness of the argmin mapping and the same
property for the (sub)level set mapping (with respect to the objective
function), for semi-infinite programs and without requiring the uniqueness
of the nominal solution, is explored too, providing an upper bound on the
calmness modulus of the argmin mapping. When confined to finitely
constrained problems, we also provide a computable upper bound as
%\sout{far as}
it only relies on the nominal data and parameters, not involving elements in
a neighborhood. Illustrative examples are provided.
\end{abstract}

\pagestyle{myheadings}
\thispagestyle{plain}

\markboth{M.J. C\'{a}novas, A.Y. Kruger, M.A. L\'{o}pez, J. Parra, M.A. Th\'era}{Calmness modulus of linear semi-infinite programs}

\section{Introduction}

We are concerned with the linear optimization problem

\begin{equation}
\begin{array}{rll}
P\left( c,b\right) : & \text{minimize } & \left\langle c,x\right\rangle  \\
& \text{subject to } & \left\langle a_{t},x\right\rangle \leq
b_{t},\,\,\,t\in T,%
\end{array}
\label{P}
\end{equation}%
where $c,x\in \mathbb{R}^{p}$ (regarded as column-vectors in matrix
calculus), $\left\langle \cdot ,\cdot \right\rangle $ denotes the usual
inner product in $\mathbb{R}^{p},T$ is a compact Hausdorff space and the
functions $t\mapsto a_{t}\in \mathbb{R}^{p}$ and $t\mapsto b_{t}\in \mathbb{R%
}$ are continuous on $T.$ {\ In the sequel we will use indifferently the
notation $\langle x,y\rangle $ or $x^{\prime }y$ to denote the scalar
product between $x$ and $y$, where in the last formula $x^{\prime }$ stands
for the transpose of $x$.} We assume that $t\mapsto a_{t}$ is a given
function and that perturbations fall on a parameter $\left( c,b\right) \in
\mathbb{R}^{p}\times C\left( T,\mathbb{R}\right) ,$ with $b\equiv \left(
b_{t}\right) _{t\in T}.$ The parameter space $\mathbb{R}^{p}\times C\left( T,%
\mathbb{R}\right) $ is endowed with the uniform convergence topology through
the norm
\begin{equation}
\left\Vert \left( c,b\right) \right\Vert :=\max \left\{ \left\Vert
c\right\Vert _{\ast },\left\Vert b\right\Vert _{\infty }\right\} ,
\label{dist cb}
\end{equation}%
where $\mathbb{R}^{p}$ is equipped with an arbitrary norm, $\left\Vert \cdot
\right\Vert $, $\left\Vert \cdot \right\Vert _{\ast }$ denotes the
associated dual norm, given by $\left\Vert u\right\Vert _{\ast
}=\max_{\left\Vert x\right\Vert \leq 1}\left\langle u,x\right\rangle ,$ and $%
\left\Vert b\right\Vert _{\infty }:=\max_{t\in T}\left\vert b_{t}\right\vert
.$ Throughout the paper $d_{\ast }$ stands for the distance associated with $%
\left\Vert \cdot \right\Vert _{\ast }.$ Our aim here is to analyze the \emph{%
solution} \emph{mapping} (also called \emph{argmin mapping}) of problem (\ref%
{P}):
\begin{equation*}
\mathcal{S}:\left( c,b\right) \mapsto \left\{ x\in \mathbb{R}^{p}\mid x\text{
solves (\ref{P}) for }\left( c,b\right) \right\} ,\text{ with }\left(
c,b\right) \in \mathbb{R}^{p}\times C\left( T,\mathbb{R}\right) .
\end{equation*}%
When $c$ is fixed, we deal with the \emph{partial} solution mapping $%
\mathcal{S}_{c}:C\left( T,\mathbb{R}\right) \rightrightarrows \mathbb{R}^{p}$
defined as
\begin{equation}
\mathcal{S}_{c}\left( b\right) :=\mathcal{S}\left( c,b\right) .  \label{Sc}
\end{equation}
\begin{definiti}
\emph{(calmness).} A mapping $S$ acting from a metric space $\left(
Y,d_{Y}\right) $ to a metric space $\left( X,d_{X}\right) $ is said to be
\emph{calm} at $\left( \bar{y},\bar{x}\right) \in \mathrm{gph}S$
(the graph of $S$) if there exist a constant $\kappa \geq 0$ and
neighborhoods $U$ of $\bar{x}$ and $V$ of $\bar{y}$ such that
\begin{equation}
d_{X}\left( x,S\left( \bar{y}\right) \right) \leq \kappa d_{Y}\left( y,%
\bar{y}\right) \text{ whenever }x\in S\left( y\right) \cap U\text{ and }%
y\in V.  \label{def_calmness}
\end{equation}
\end{definiti}
Recall that calmness of $S$ turns out
to be equivalent to  \emph{metric subregularity of }$S^{-1}$ (see \cite[Theorem 3H.3 and Exercise 3H.4]{DR09}) 
 which reads in terms of the existence of a
(possibly smaller) neighborhood $U$ of $\bar{x}$ and $\kappa \geq 0$
such that 
\begin{equation}
d_X\left( x,S\left( \bar{y}\right) \right) \leq \kappa d_Y\left( \bar{y%
},S^{-1}\left( x\right) \right) ,\text{ for all }x\in U,  \label{eq_msubreg}
\end{equation}%
where $S^{-1}\left( x\right) :=\left\{ y\in Y\mid x\in S\left( y\right)
\right\} .$
The variant of the previous definition when the point-to-set distance $%
d_{X}\left( x,S\left( \bar{y}\right) \right) $ is replaced with $%
d_{X}\left( x,\bar{x}\right) $ in (\ref{def_calmness}) corresponds to
the so-called \emph{isolated calmness} property. Isolated calmness of $S$ at
$\left( \bar{y},\bar{x}\right) \in \mathrm{gph}S$ implies that $%
S\left( \bar{y}\right) \cap U=\{\bar{x}\},$ so $\bar{x}$ is
an isolated point in $S\left( \bar{y}\right) ,$ hence the terminology.
In fact, the reader can observe that isolated calmness (also called \emph{%
calmness on selections}, see e.g. \cite{HeJouOu02}, or \emph{local upper
Lipschitz }property, see e.g. \cite{KlaKu02}) is nothing else but standard
calmness together with this isolatedness condition (i. e., $S\left(
\bar{y}\right) \cap U=\{\bar{x}\}$ for some neighborhood $U$ of $%
\bar{x}$). Observe that in this case $S\left( y\right) \cap U$ needs
not to be a singleton and might be empty for $y\in V\backslash \{\bar{y}%
\}$. In the case of our convex-valued argmin mapping $\mathcal{S},$ an
isolated solution is nothing else but a unique solution. Sometimes along the
paper we will require this uniqueness assumption of the nominal problem $%
P\left( \bar{c},{\bar{b}}\right) .$
Any linear mapping $A:\mathbb{R}^{m}\rightarrow \mathbb{R}^{n}$ is
isolatedly calm at any point while isolated calmness of {\ the inverse
mapping } $A^{-1}$ is equivalent to injectivity of $A$, that is, $%
A^{-1}(0)=\{0\}.$ More generally, from a result by Robinson \cite{Robinson}
it follows that a set-valued mapping $S:\mathbb{R}^{m}\rightrightarrows
\mathbb{R}^{n}$ whose graph is the union of finitely many polyhedral convex
sets is isolatedly calm {\ at $(\bar{y},\bar{x})$} if and only if $%
\bar{x}$ is an isolated point of $S\left( \bar{y}\right) .$ For a
function $f:\mathbb{R}^{m}\rightarrow \mathbb{R}^{n}$ which is smooth in a
neighborhood of $\bar{x}$, the inverse $f^{-1}$ is isolatedly calm at {%
\ $(f(\bar{x}),\bar{x})$} if and only if the derivative mapping $Df(%
\bar{x})$ is injective. As shown in \cite{Dontchev95}, see also \cite%
{ample}, this inverse-function-type result can be extended to mappings of
the form $f+F$ where $f$ is a smooth function and $F$ is a set-valued
mapping with closed values. In \cite{ample} it was shown that, for a problem
of minimizing a convex function with linear perturbations over a polyhedral
convex set, the property of isolated calmness of the solution set is \emph{%
equivalent} to the standard second-order sufficient optimality condition. In
our semi-infinite framework, the set of feasible solutions is generally
\emph{not} polyhedral.

Calmness is known to be a{\ weakened version of the more robust \emph{Aubin
property}}: We say that $S$ is \emph{Aubin continuous} (or {\ pseudo
Lipschitz}, or enjoys the Aubin property) at $\left( \bar{y},\bar{x%
}\right) \in \mathrm{gph}S$ if there exist $\kappa \geq 0$ and neighborhoods
$V$ of $\bar{y}\ $and $U$ of $\bar{x}$ such that
\begin{equation}
d_{X}\left( x,S\left( y^{\prime }\right) \right) \leq \kappa d_{Y}\left(
y,y^{\prime }\right) \text{ whenever }y,y^{\prime }\in V\text{ and }x\in
S\left( y\right) \cap U.  \label{Aubin}
\end{equation}%
{\ Calmness} corresponds to the case when we fix $y^{\prime }=\bar{y}$
in (\ref{Aubin}).

The infimum of the values of $\kappa $ for which (\ref{Aubin}) holds (for
 some associated neighborhoods $U$ and $V$) is called the \emph{Lipschitz
modulus} of $S$ at $(\bar{y},\bar{x})$, and denoted by $\mathrm{lip%
}S\left( \bar{y},\bar{x}\right) .$ Concerning calmness, the
infimum of the values of $\kappa $ for which (\ref{def_calmness}) holds (for
some associated neighborhoods $U$ and $V$) is called the \emph{calmness
modulus} of $S$ at $(\bar{y},\bar{x})$, and denoted by $\mathrm{clm%
}S\left( \bar{y},\bar{x}\right) .$ As a direct consequence of the
definitions we have
\begin{eqnarray}
\mathrm{lip}S\left( \bar{y},\bar{x}\right)  &=&\limsup_{\substack{ %
y,y^{\prime }\rightarrow \bar{y},~y\neq y^{\prime } \\ x\rightarrow \bar{x}%
,~x\in S\left( y\right) }}\frac{d_{X}\left( x,S\left( y^{\prime }\right)
\right) }{d_{Y}\left( y,y^{\prime }\right) },  \label{formulae lip} \\
\mathrm{clm}S\left( \bar{y},\bar{x}\right)  &=&\limsup_{\substack{ %
y\rightarrow \bar{y} \\ x\rightarrow \bar{x},~x\in S\left( y\right) }}\frac{%
d_{X}\left( x,S\left( \bar{y}\right) \right) }{d_{Y}\left( y,\bar{y}\right) }%
.  \label{formulae clm}
\end{eqnarray}%
Accordingly, we always have $\mathrm{clm}S\left( \bar{y},\bar{x}%
\right) \leq \mathrm{lip}S\left( \bar{y},\bar{x}\right) .$
It
is well-known that $\mathrm{clm}S\left( \bar{y},\bar{x}\right) $
coincides with the \emph{modulus of metric subregularity of }$S^{-1}$ at $%
\left( \bar{x},\bar{y}\right) $; i.e., with the infimum of the
values of $\kappa $ for which (\ref{eq_msubreg}) holds (for   associated
neighborhoods $U$). So, we can write 
$$
\mathrm{clm}S\left( \bar{y},\bar{x}\right) =\limsup_{x\rightarrow 
\bar{x}}\frac{d_{X}\left( x,S\left( \bar{y}\right) \right) }{%
d_{Y}\left( \bar{y},S^{-1}\left( x\right) \right) }.
\label{eq_modulus_subreg}
$$
The main goal of this paper consists in computing or estimating $\mathrm{clm}\mathcal{S}\left( \left( {\bar{c}},{\bar{b}}\right) ,%
\bar{x}\right) ,$ assuming sometimes that $\mathcal{S}\left( {\bar{c}}%
,{\bar{b}}\right) =\left\{ \bar{x}\right\} $ and the Slater
constraint qualification holds at $P\left( {\bar{c}},{\bar{b}}\right) $
(see {\ Section} 2). As emphasized in \cite{HeJouOu02}, the calmness
property plays a key role in optimization theory, specifically in the study
of optimality conditions, sensitivity analysis, stability of solutions or
existence of error bounds. The calmness modulus in the context of constraint
systems under right-hand side perturbations has been widely analyzed in the
literature (the reader is addressed to \cite{KNT10} for more details). In
particular, \cite[Theorem 1]{KNT10} constitutes a starting point for the
present paper, as far as it provides a formula for the calmness modulus of
constraint systems in terms of the associated supremum function.

Specifically, the structure of the paper is as follows: {\ Section} 2
introduces the necessary notation and background. {\ Section} 3 provides a
lower bound on the calmness modulus of $\mathcal{S}$ for semi-infinite
programs with unique optimal solutions. {\ Section} 4 shows that this lower
bound equals the exact modulus when $T$ is finite. The uniqueness of optimal
solution is not required in {\ Section} 4 for establishing the upper
estimate. {\ Section} 5 is concerned with upper bounds on $\mathrm{clm}%
\mathcal{S}\left( \left( {\bar{c}},{\bar{b}}\right) ,\bar{x}\right)
.$ The first part is devoted to explore the relationship between the moduli
of $\mathcal{S}$ and the (sub)level set mapping $\mathcal{L}$ introduced in
\cite{CHPT12} for characterizing the calmness of $\mathcal{S}$ at $\left(
\left( {\bar{c}},{\bar{b}}\right) ,\bar{x}\right) $ (see {\ Section}
2). The second part of {\ Section} 5 is confined to finitely constrained
problems under uniqueness of optimal solution and obtains a new upper bound
which may be easily computable as it is formulated exclusively in terms of
the nominal data ${\bar{c}}$, ${\bar{b}}$, and $\bar{x}$.  Section 6
 provides examples showing that the latter upper bound may not
be attained and also illustrating how the expression for the exact modulus
(for finitely constrained problems with unique optimal solutions) given in {%
\ Section} 4 works. Finally in Section 7 we give some concluding remarks.

\section{Preliminaries}

In this section we introduce some notation and preliminary results. Given $%
X\subset \mathbb{R}^{k},$ $k\in \mathbb{N},$ we denote by \textrm{int}$X$,
\textrm{co}$X,$ and \textrm{cone}$X$ the \textit{\ interior}, the \emph{%
convex hull}, and the \emph{conical convex hull }of $X$, respectively. It is
assumed that \textrm{cone}$X$ always contains the zero-vector {$0$}$_{k}$,
in particular \textrm{cone}$(\varnothing )=\{${$0$}$_{k}\}.$ Associated with
the parameterized problem (\ref{P}), we denote by $\mathcal{F}$ the feasible
set mapping, which is given by
\begin{equation*}
\mathcal{F}\left( b\right) :=\left\{ x\in \mathbb{R}^{p}\mid \text{%
\thinspace }\left\langle a_{t},x\right\rangle \leq b_{t},\,\,\,t\in
T\right\} .
\end{equation*}%
We also {remind that the \emph{set of active indices} and the \emph{active
cone} at $x\in \mathcal{F}\left( b\right) $ are the sets $T_{b}\left(
x\right) $ and $A_{b}\left( x\right) $ defined respectively by}
\begin{equation*}
T_{b}\left( x\right) :=\left\{ t\in T\mid \text{\thinspace }\left\langle
a_{t},x\right\rangle =b_{t}\right\} \text{ and }
\end{equation*}%
\begin{equation*}
A_{b}\left( x\right) :=\mathrm{cone}\left\{ a_{t}\mid t\in T_{b}\left(
x\right) \right\} .\medskip
\end{equation*}

Recall that the \emph{Slater constraint qualification }(hereafter called
\emph{Slater condition}) holds {\ for} the problem $P\left( c,b\right) $ if
there exists $\widehat{x}\in \mathbb{R}^{p}$ such that $a_{t}^{\prime }%
\widehat{x}<b_{t}$ for all $t\in T.$ The following result can be traced out
from \cite{libro}:

\begin{prop}
\label{Prop_CLFM} Let $\left( {\bar{c}},{\bar{b}}\right) \in \mathbb{R}%
^{p}\times C\left( T,\mathbb{R}\right) $ and assume that $P\left( {\bar{c}},{%
\bar{b}}\right) $ satisfies the Slater condition. Then $\bar{x}\in
\mathcal{S}\left( {\bar{c}},{\bar{b}}\right) $ if and only if the \emph{%
Karush-Kuhn-Tucker (KKT) }condition holds, i. e.,
\begin{equation}
\bar{x}\in \mathcal{F}\left( {\bar{b}}\right) \text{ and \ }-{\bar{c}}%
\in A_{{\bar{b}}}\left( \bar{x}\right) {\normalsize .}  \label{KKTcond}
\end{equation}
\end{prop}

In the next result we appeal to the concept of strong uniqueness of
minimizers. We say that $\bar{x}\in \mathbb{R}^{p}$ is the \emph{%
strongly\ unique minimizer }of $P\left( {\bar{c}},{\bar{b}}\right) $ if for
some $\alpha >0$ we have
\begin{equation}
\left\langle {\bar{c}},x\right\rangle \geq \left\langle {\bar{c}},\bar{x}%
\right\rangle +\alpha \left\Vert x-\bar{x}\right\Vert \text{ for all }%
x\in \mathcal{F}\left( {\bar{b}}\right) .  \label{ShM}
\end{equation}

\begin{theo}
\emph{\cite[Thm. 3]{CDLP09}}\label{th charact_linear} Consider the
parameterized linear optimization problem \emph{(\ref{P})} and let $\mathcal{%
S}$ be the associated solution mapping. Let $\left( \left( {\bar{c}},{\bar{b}%
}\right) ,\bar{x}\right) \in \mathrm{gph}\mathcal{S}$ be such that $%
P\left( {\bar{c}},{\bar{b}}\right) $ satisfies the Slater{\ condition}. Then
the following are equivalent:

\emph{(i) }$-{\bar{c}}\in $\textrm{int}\thinspace $A_{{\bar{b}}}\left(
\bar{x}\right) ;$

\emph{(ii) }$\bar{x}$\emph{\ }is the strongly unique minimizer of $%
P\left( {\bar{c}},{\bar{b}}\right) ;$

\emph{(iii) }$\mathcal{S}$ is isolatedly calm at $\left( \left( {\bar c},{%
\bar b}\right) ,\bar{x}\right) ;$

\emph{(iv) }$\mathcal{S}_{{\bar c}}$ \emph{(}defined in \emph{(}$\emph{\ref%
{Sc})})$ is isolatedly calm at $\left( {\bar b},\bar{x}\right) ; $

If $T$ is finite, we can add the following condition:

\emph{(v) }$\bar{x}$\emph{\ }is the unique solution of $P( {\bar c},{%
\bar b}).$
\end{theo}

\begin{rem}
\label{Rem chain}\emph{Paper \cite{CDLP09} deals with convex optimization
problems with canonical perturbations, and provides examples showing that
implications }%
\begin{equation*}
\emph{\ (ii)\ }\Rightarrow \emph{\ (iii)\ }\Rightarrow \emph{\ (iv)\ }%
\Rightarrow \emph{\ (v)}
\end{equation*}%
\emph{may be strict in the convex case. This paper also extends condition
(i) to the convex case and shows the equivalence between (ii) and this
extended version of (i).}
\end{rem}

Concerning the Aubin property, we have the following result (where $%
\left\vert D\right\vert $ stands for the cardinality of $D$):

\begin{theo}
\emph{\cite[Theorem 16]{CKLP}}\label{RM3} For the linear semi-infinite
program (\ref{P}), let $\left( \left( {\bar c},{\bar b}\right) ,%
\bar{x}\right) \in \mathrm{gph}\mathcal{S}.$ Then, the following
conditions are equivalent: \vspace{0.5ex}

(i) $\mathcal{S}$ is Aubin continuous at $\left( \left( \bar{c},\bar{b}%
\right) ,\bar{x}\right);$

(ii) $\mathcal{S}$ is strongly Lipschitz stable at $\left( \left( \bar{c},%
\bar{b}\right) ,\bar{x}\right) $ (i. e., locally single-valued and
Lipschitz continuous around $\left( \bar{c},\bar{b}\right) $)$;$

(iii)$\mathcal{S}$ is locally single-valued and continuous in some neighborhood of $(%
\bar{c},\bar{b});$

(iv) $\mathcal{S}$ is single valued in some neighborhood of $\left( \bar{c},%
\bar{b}\right) ;$

(v) $P\left( \bar{c},\bar{b}\right) $ satisfies the Slater{\ condition} and
there is no $D\subset T_{\bar{b}}\left( \bar{x}\right) $ with $%
\left\vert D\right\vert <p$ such that $-\bar{c}\in \mathrm{cone}\left(
\left\{ a_{t},\text{ }t\in D\right\} \right) ;$

(vi) $P\left( \bar{c},\bar{b}\right) $ satisfies the Slater{\ condition} and
for each $D\subset T_{\bar{b}}\left( \bar{x}\right) $ with $\left\vert
D\right\vert =p$ such that $-\bar{c}\in \mathrm{cone}\left( \left\{ a_{t},%
\text{ }t\in D\right\} \right) $, all possible subsets of $\{a_{t},$ $t\in
D;-{\bar{c}}\}$ with $p$ elements are linearly independent;

(vii) $\left( \bar{c},\bar{b}\right) \in \mathrm{{int}\thinspace}\left(
\{\left( c,b\right) \mid \mathcal{S}\left( c,b\right) \text{ consists of a
strongly unique minimizer}\}\right) .$
\end{theo}

We emphasize condition $(v)$ which will be referred to as the \emph{N\"{u}%
rnberger condition}, since it was stated in \cite{Nu84} (in the equivalent
form $\left( vi\right) $) for characterizing the counterpart of condition $%
\left( vii\right) $ when all coefficients (including the $a_{t}$'s) are
subject to perturbations.

In the next theorem we provide a characterization of the calmness of $%
\mathcal{S}$ at $\left( \left( {\bar{c}},{\bar{b}}\right) ,\bar{x}%
\right) ,$ under {\ the} Slater{\ condition}, without requiring $\mathcal{S}%
\left( {\bar{c}},{\bar{b}}\right) $ to be reduced to the singleton $\{%
\bar{x}\}.$ In this characterization we {use} the \emph{level set
mapping} $\mathcal{L}:\mathbb{R\times }C\left( T,\mathbb{R}\right)
\rightrightarrows \mathbb{R}^{p}$ given by
\begin{equation}
\mathcal{L}\left( \alpha ,b\right) =\left\{ x\in \mathbb{R}^{p}\mid {\bar{c}%
^{\prime }x}\leq \alpha ;\text{ }{a_{t}^{\prime }x}\leq b_{t},\,\,t\in
T\right\}   \label{eq_levelset}
\end{equation}%
and the supremum function $\bar{f}:\mathbb{R}^{p}\rightarrow \mathbb{R}$
defined as %\marg{I suggest $\bar c'(x-\bar{x})$}
\begin{equation}
\bar{f}\left( x\right) :=\sup \left\{ {\bar{c}^{\prime }(x-\bar{x})%
};\text{ }a_{t}^{\prime }x-{\bar{b}}_{t},\,\,\,t\in T\right\} .
\label{eq_f_bar}
\end{equation}
In the following theorem, we appeal to the
zero sublevel set of $\bar{f}:$%
\begin{equation}
\left[ \bar{f}\leq 0\right] :=\left\{ x\in \mathbb{R}^{p}\mid 
\bar{f}\left( x\right) \leq 0\right\} =\mathcal{L}\left( \bar{c}%
^{\prime }\bar{x},\bar{b}\right) =\mathcal{S}\left( \bar{c},%
\bar{b}\right) . \label{eq_sublevel}
\end{equation}

%Theorem 3

\begin{theo}
\label{Th_charact calmness}Let $(({\bar{c}},{\bar{b}}),\bar{x})\in
\mathrm{gph}\mathcal{S}$ and assume $P\left( \bar{c}\bar{,b}\right) $
satisfies the Slater{\ condition}. Then the following conditions are
equivalent:

$(i)$\emph{\ }$\mathcal{S}$ is calm at $(({\bar c},{\bar b}),\bar{x}%
);$

$(ii)$\emph{\ }$\mathcal{S}_{{\bar c}}$ is calm at $({\bar b},\bar{x%
});$

$(iii)$\emph{\ }$\mathcal{L}$ is calm at $((\left\langle {\bar c},%
\bar{x}\right\rangle ,{\bar b}),\bar{x}){\ ;}$

$\left( iv\right) $ $\bar{f}$ has a \emph{local error bound} at $%
\bar{x};$\ i. \hspace{-0.15cm}e., there exist $\kappa \geq 0$ and a
neighborhood $U$ of $\bar{x}$ such that%
\begin{equation*}
d\left( x,\left[ \bar{f}\leq 0\right] \right) \leq \kappa \left[
\bar{f}\left( x\right) \right] _{+},\text{ for all }x\in U{\ ;}
\end{equation*}%
{\ (we recall that given a real $a$, %\sout{we note}
$a_{+}$ {stands} for $\max (a,0)$)}

$\left( v\right) $ $\underset{x\rightarrow \bar{x},~\bar{f}\left(
x\right) >0}{\lim \inf }d_{\ast }\left( {0}_{p},\partial \bar{f}\left(
x\right) \right) >0.$
\end{theo}

\begin{proof}
The equivalence $\left( i\right) \Leftrightarrow \left( ii\right)
\Leftrightarrow \left( iii\right) $ is exactly \cite[Theorem 1]{CHPT12},
whereas $\left( iii\right) \Leftrightarrow \left( iv\right) $ is a direct

consequence of the definitions, as far as, 
\begin{eqnarray}
\left[ \bar{f}\left( x\right) \right] _{+} &=&\left[ \sup \{\bar{c}%
^{\prime }x-\bar{c}^{\prime }\bar{x};\text{ }a_{t}^{\prime }x-%
\bar{b}_{t},\,t\in T\}\right] _{+}  \label{eq_positive_part}\nonumber \\
&=&\sup \{\left[ \bar{c}^{\prime }x-\bar{c}^{\prime }\bar{x}%
\right] _{+};\text{ }\left[ a_{t}^{\prime }x-\bar{b}_{t}\right]
_{+},\,t\in T\}   \\
&=&d\left( \left( \bar{c}^{\prime }\bar{x},\bar{b}\right) ,%
\mathcal{L}^{-1}\left( x\right) \right) ,\text{ for all }x\in \mathbb{R}^{p},\nonumber
\end{eqnarray}%
taking also into account (\ref{eq_sublevel}) and the equivalence between the
calmness of $\mathcal{L}$ and the metric subregularity of $\mathcal{L}^{-1}$
(see (\ref{eq_msubreg})).

 Finally, $\left( i\right) \Leftrightarrow
\left( v\right) $ can be traced out from \cite[Proposition 2.1 and Theorem
5.1]{AzCo04} (see also \cite{Jou00}), and constitutes a key tool for
establishing $\left( ii\right) \Rightarrow \left( iii\right) $ in \cite[%
Theorem 1]{CHPT12}.
\end{proof}

In order to provide an upper bound on the calmness modulus in {\ Section} 5,
we use the following notation: {\ for any finite} subset of indices $D\subset T, A_{D}$ 
denotes the matrix whose rows are $a_{t}^{\prime }, t\in D$ (given in some prefixed order).
%$ where {\ the symbol }$^{\prime }$ {%
%\ stands for the} transposition (recall that elements of $\mathbb{R}^{p}$
%are regarded as column-vectors). 
Given any $b\in C\left( T,\mathbb{R}\right),$ we denote $b_{D}=\left( b_{t}\right) _{t\in D}.$ If moreover $x\in
\mathcal{S}\left( {\bar{c}},b\right) $, we {de}{note}
\begin{equation}
\mathcal{T}_{b}\left( x\right) =\left\{ D\subset T_{b}\left( x\right)
\left\vert
\begin{array}{c}
\left\vert D\right\vert =p,\text{ }A_{D}\text{ is nonsingular,} \\
\text{and }-{\bar{c}}\in \mathrm{cone}\left\{ a_{t},\text{ }t\in D\right\}
\end{array}%
\right. \right\} .  \label{eq_tbx}
\end{equation}%
This set was already introduced in \cite[p. 520]{CGP08} with the aim of
obtaining bounds on the Lipschitz modulus, provided that the Aubin property
holds (i. e., under the N\"{u}rnberger condition). In that paper, the
nonsingularity of $A_{D}$ was not explicitly required when defining $%
\mathcal{T}_{b}\left( x\right) ,$ since it follows as a consequence of the N%
\"{u}rnberger condition (at points $\left( \left( {\bar{c}},b\right)
,x\right) $ $\in \mathrm{gph}\mathcal{S}$ around $\left( \left( {\bar{c}},{%
\bar{b}}\right) ,\bar{x}\right) $).

Given $D\subset T$ with $\left\vert D\right\vert =p,$ we identify {\ the}
matrix $A_{D}$ with the `endomorphism' $\mathbb{R}^{p}\ni x\mapsto A_{D}x\in
\mathbb{R}^{D}$, where $\mathbb{R}^{p}$ is equipped with an arbitrary norm $%
\left\Vert \cdot \right\Vert $ and $\mathbb{R}^{D}$ is endowed with the
supremum norm $\left\Vert \cdot \right\Vert _{\infty }.$ Recall that $\left(
\left\Vert \cdot \right\Vert _{\infty }\right) _{\ast }=\left\Vert \cdot
\right\Vert _{1}.$ For our choice of norms, provided that $A_{D}$\ is
non-singular, we have
\begin{equation}
\left\Vert A_{D}^{-1}\right\Vert :=\max_{\left\Vert y\right\Vert _{\infty
}\leq 1}\left\Vert A_{D}^{-1}y\right\Vert =\max_{y\in \left\{ -1,1\right\}
^{p}}\left\Vert A_{D}^{-1}y\right\Vert =\left( \min_{\left\Vert \lambda
\right\Vert _{1}=1}\left\Vert A_{D}^{\prime }\lambda \right\Vert _{\ast
}\right) ^{-1}.  \label{Norm calculus}
\end{equation}%
The second equality comes from the use of $\left\Vert \cdot \right\Vert
_{\infty }$ in $\mathbb{R}^{D},$ together with the fact that\emph{\ }$%
\left\{ -1,1\right\} ^{p}$ is the set of extreme points of the associated
closed unit ball and the function to be maximized is convex. The last
equality is a straightforward consequence of \cite[Corollary 3.2]{CDLP05}
together with the fact that $\left\Vert A_{D}^{-1}\right\Vert $ coincides
with the (metric) regularity modulus of $A_{D}$ (Lipschitz modulus of $%
A_{D}^{-1}$), at any point of its graph.

%Theorem 4

\begin{theo}
\emph{(see \cite[Theorems 1 and 2, and Corollary 2]{CGP08})}\label%
{Th_cotainf} Assume, for the linear semi-infinite program (\ref{P}), that
the N\"{u}rnberger condition holds at $\left( \left( {\bar{c}},{\bar{b}}%
\right) ,\bar{x}\right) \in \mathrm{gph}\mathcal{S}.$ Then
\begin{equation}
\mathrm{lip}\mathcal{S}\left( \left( {\bar{c}},{\bar{b}}\right) ,\bar{x}%
\right) \geq \sup_{D\in \mathcal{T}_{{\bar{b}}}(\bar{x})}\left\Vert
A_{D}^{-1}\right\Vert ,  \label{CotaInf}
\end{equation}%
and equality occurs in (\ref{CotaInf}) provided that either $T$ is finite or
$p\leq 3$.
\end{theo}

In \cite{CGP08}, a general sufficient condition --called {\ $(\mathcal{H})$}
therein-- ensuring equality in (\ref{CotaInf}) was introduced. Condition {\ $%
(\mathcal{H})$} always holds when either $T$ is finite or $p\leq 3$.
Examples in \cite{CGP08} show that condition {\ $(\mathcal{H})$} can fail in
$\mathbb{R}^{4}$. The question of whether {\ $(\mathcal{H})$} can be
weakened (or even removed) while still ensuring equality in (\ref{CotaInf})
when $T$ is infinite remains an open problem.
%More in detail, \cite{CGP08} introduces a sufficient condition for the
%equality in (\ref{CotaInf}) --called { $(\mathcal{H})$} therein-- which always holds when $%
%n\leq 3,$ providing examples in $\mathbb{R}^{4}$ where this condition holds
%and others where it fails. The problem of whether { $(\mathcal{H})$} may be weakened or
%even removed to ensure equality in (\ref{CotaInf}) when $T$ is infinite
%remains as an open problem.

In the case when $T$ is finite, it is obvious that, under the assumptions of
Theorem \ref{Th_cotainf}, $\sup_{D\in \mathcal{T}_{{\bar{b}}}(\bar{x}%
)}\left\Vert A_{D}^{-1}\right\Vert $ is an upper bound on $\mathrm{clm}%
\mathcal{S}\left( \left( {\bar{c}},{\bar{b}}\right) ,\bar{x}\right) .$
%At this moment we advance (see {\S 5})
We show in {\ Section} 5 that, under the Slater{\ condition} and assuming
that $T$ is finite and $\mathcal{S}\left( {\bar{c}},{\bar{b}}\right)
=\left\{ \bar{x}\right\} ,$ $\sup_{D\in \mathcal{T}_{{\bar{b}}}(%
\bar{x})}\left\Vert A_{D}^{-1}\right\Vert $ %still works as
remains an upper bound of $\mathrm{clm}\mathcal{S}\left( \left( {\bar{c}},{%
\bar{b}}\right) ,\bar{x}\right) $ %in the case
even when $\mathcal{S}$ is not Aubin continuous at $\left( \left( \bar{c},%
\bar{b}\right) ,\bar{x}\right) .$ Let us point out that, when $T$ is
finite, $\mathcal{S}$ is always calm at any point of its graph as {it was
proved in} \cite{Robinson}\ (see \cite[Section 4]{CHPT12} for more
%deatils).
details).

\section{Lower bound on the calmness modulus}

%Along
In this section and the next one, we consider, associated with $\left(
b,x\right) \in \mathrm{gph}\mathcal{S}_{{\bar{c}}}$, the family of
\textquotedblleft KKT subsets" of $T$ given by
\begin{equation*}
\mathcal{K}_{b}\left( x\right) =\left\{ D\subset T_{b}\left( x\right)
\left\vert \left\vert D\right\vert \leq p\text{ and }-{\bar{c}}\in \mathrm{%
cone}\left\{ a_{t},\text{ }t\in D\right\} \right. \right\} ,
\end{equation*}%
which coincides with $\mathcal{T}_{b}\left( x\right) $ under the N\"{u}%
rnberger condition at $\left( \left( \bar{c},b\right) ,x\right) .$ For any $%
D\in \mathcal{K}_{{\bar{b}}}\left( \bar{x}\right) ,$ we consider the
supremum function, $f_{D}:\mathbb{R}^{p}\rightarrow \mathbb{R}$ given by%
\begin{align*}
f_{D}\left( x\right) & :=\sup \left\{ \left\langle a_{t},x\right\rangle -{%
\bar{b}}_{t},\text{ }t\in T;\text{ }-\left\langle a_{t},x\right\rangle +{%
\bar{b}}_{t},\text{ }t\in D\right\}  \\
& =\sup \left\{ \left\langle a_{t},x\right\rangle -{\bar{b}}_{t},\text{ }%
t\in T\setminus D;\text{ }\left\vert \left\langle a_{t},x\right\rangle -{%
\bar{b}}_{t}\right\vert ,\text{ }t\in D\right\} .
\end{align*}%
Roughly speaking, for a given $x\in \mathbb{R}^{p},$ $f_{D}\left( x\right) $
is the smallest perturbation size on ${\bar{b}},$ say $\left\Vert b-{\bar{b}}%
\right\Vert _{\infty }$, that makes $x$ a KKT point (hence optimal) for
problem $P\left( {\bar{c}},b\right) $ with $D$ as an associated
\textquotedblleft KKT index set". In the semi-infinite framework some
technical arrangements are needed to %get
obtain a continuous perturbation. %This
These comments will be formalized when constructing a certain sequence $%
\left\{ b^{n}\right\} _{n\in \mathbb{N}}$ in the following theorem.

For each $D\in \mathcal{K}_{{\bar{b}}}\left( \bar{x}\right) $ we may
consider the {set-valued} mapping $\mathcal{L}_{D}:C\left( T,\mathbb{R}%
\right) \times \mathbb{R}^{D}\rightrightarrows \mathbb{R}^{p}$ given by
\begin{equation*}
\mathcal{L}_{D}\left( b,d\right) :=\left\{ x\in \mathbb{R}^{p}\mid
\left\langle a_{t},x\right\rangle \leq b_{t},~t\in T;~\left\langle
-a_{t},x\right\rangle \leq d_{t},~t\in D\right\} .
\end{equation*}%
Observe that $\mathcal{L}_{D}\left( {\bar{b}},-{\bar{b}}_{D}\right) ,$ where
${\bar{b}}_{D}=\left( {\bar{b}}_{t}\right) _{t\in D},$ is the referred set
of KKT points of problem $P\left( {\bar{c}},{\bar{b}}\right) $ associated
with $D$ as the KKT index set. Accordingly,
\begin{equation}
\mathcal{L}_{D}\left( {\bar{b}},-{\bar{b}}_{D}\right) {=[f_{D}=0]}\subset
\mathcal{S}\left( {\bar{c}},{\bar{b}}\right) \text{ for all }D\in \mathcal{K}%
_{{\bar{b}}}\left( \bar{x}\right) .  \label{2001}
\end{equation}%
Let us also observe that
%\begin{equation}
%\mathrm{clm}\mathcal{L}_{D}\left( \left( {\bar{b}},-{\bar{b}}_{D}\right)
%\mid \bar{x}\right) =\limsup_{\substack{ x\rightarrow \bar{x}  \\ %
%f_{D}\left( x\right) >0}}\frac{1}{d_{\ast }\left( {0}_{p},\partial
%f_{D}\left( x\right) \right) },  \label{2001,5}
%\end{equation}%
%according to \cite[Theorem 1]{KNT10}.
\begin{eqnarray}
\mathrm{clm}\mathcal{L}_{D}\left( \left( {\bar{b}},-{\bar{b}}_{D}\right) ,%
\bar{x}\right)  &=&\limsup_{x\rightarrow \bar{x}}\frac{d\left( x,%
\mathcal{L}_{D}\left( {\bar{b}},-{\bar{b}}_{D}\right) \right) }{d\left(
\left( {\bar{b}},-{\bar{b}}_{D}\right) ,\mathcal{L}_{D}^{-1}\left( x\right)
\right) }  \label{2001,5} \\
&=&\limsup_{x\rightarrow \bar{x}}\frac{d\left( x,\left[ f_{D}\leq 0%
\right] \right) }{\left[ f_{D}\left( x\right) \right] _{+}}  \notag \\
&=&\limsup_{\substack{ x\rightarrow \bar{x} \\ f_{D}\left( x\right) >0}}%
\frac{1}{d_{\ast }\left( {0}_{p},\partial f_{D}\left( x\right) \right) }, 
\notag
\end{eqnarray}%
where the last equality comes from \cite[Theorem 1]{KNT10} (and the second
one appeals to a similar argument as in (\ref{eq_positive_part})).\bigskip 
\begin{theo}
\label{Th_lowerbound}Let $\mathcal{S}\left( {\bar{c}},{\bar{b}}\right)
=\left\{ \bar{x}\right\} $ and assume that $P\left( {\bar{c}},{\bar{b}}%
\right) $ satisfies the Slater{\ condition}. Then%
\begin{equation*}
\mathrm{clm}\mathcal{S}\left( \left( {\bar{c}},{\bar{b}}\right) ,\bar{x}%
\right) \geq \mathrm{clm}\mathcal{S}_{{\bar{c}}}\left( {\bar{b}},\bar{x}%
\right) \geq \sup_{D\in \mathcal{K}_{{\bar{b}}}\left( \bar{x}\right)
}\limsup_{\substack{ x\rightarrow \bar{x}  \\ f_{D}\left( x\right) >0}}%
\frac{1}{d_{\ast }\left( {0}_{p},\partial f_{D}\left( x\right) \right) }.
\end{equation*}
\end{theo}

\begin{proof}
The first inequality follows straightforwardly from (\ref{formulae clm}). In
order to prove the second one, let us consider a fixed $D\in \mathcal{K}_{{%
\bar{b}}}\left( \bar{x}\right) $ and let us %see
show that
\begin{equation}
\mathrm{clm}\mathcal{S}_{{\bar{c}}}\left( {\bar{b}},\bar{x}\right) \geq
\limsup_{\substack{ x\rightarrow \bar{x}  \\ f_{D}\left( x\right) >0}}%
\frac{1}{d_{\ast }\left( {0}_{p},\partial f_{D}\left( x\right) \right) }.
\label{eq:clnSc>=limsup}
\end{equation}%
Write
\begin{equation*}
\limsup_{\substack{ x\rightarrow \bar{x}  \\ f_{D}\left( x\right) >0}}%
\frac{1}{d_{\ast }\left( {0}_{p},\partial f_{D}\left( x\right) \right) }%
=\lim_{n\rightarrow +\infty }\frac{1}{d_{\ast }\left( {0}_{p},\partial
f_{D}\left( x^{n}\right) \right) },
\end{equation*}%
for a certain sequence $\left\{ x^{{n}}\right\} _{n\in \mathbb{N}}$ such
that $\displaystyle\lim_{{n\rightarrow +\infty }}x^{{n}}=\bar{x}$ and $%
f_{D}\left( x^{{n}}\right) >0$ for all ${n}\in \mathbb{N}.$ Obviously, $x^{{n%
}}\notin \mathcal{S}_{{\bar{c}}}\left( {\bar{b}}\right) $ since $\mathcal{S}%
_{{\bar{c}}}\left( {\bar{b}}\right) =\left\{ \bar{x}\right\} $ and $%
f_{D}\left( \bar{x}\right) =0.$ Note that $d_{\ast }\left( {0}%
_{p},\partial f_{D}\left( x^{{n}}\right) \right) >0$ since $x^{{n}}\notin
\arg \min_{x\in \mathbb{R}^{p}}f_{D}.$
%\marginpar{\flushleft\small{A reference to the book by Ioffe-Tikhomirov would be more appropriate. Unfortunately, I only have its original Russian version. Anyway, I will try to check the reference when I come back to my office in January.}}
From now on, the proof is focused on the construction of a new sequence of
parameters $\left\{ b^{n}\right\} \subset C\left( T,\mathbb{R}^{p}\right) $
such that
\begin{equation}
x^{n}\in \mathcal{S}_{\bar{c}}(b^{n})\text{ and }\frac{\Vert x^{n}-\bar{x}%
\Vert }{\Vert b^{n}-\bar{b}\Vert _{\infty }}\geq \frac{n}{n+1}\cdot \frac{1}{%
d_{\ast }({0}_{p},\partial f_{D}(x^{n})}.  \label{eq_proof_fD}
\end{equation}%
Firstly, we will provide a lower bound on $\left\Vert x^{n}-\bar{x}%
\right\Vert .$ Clearly (\ref{eq_proof_fD}) implies (\ref{eq:clnSc>=limsup}).
%According to
In view of the closedness of $\partial f_{D}\left( x^{n}\right) $, we can
write $d_{\ast }\left( {0_{p}},\partial f_{D}\left( x^{n}\right) \right)
=\left\Vert u^{n}\right\Vert _{\ast }$ for a certain
\begin{equation*}
\begin{tabular}{ll}
$u^{n}\in \partial f_{D}\left( x^{n}\right) =\mathrm{conv}$ & $\hspace{-0.4cm%
}\left( \left\{ a_{t}\mid \left\langle a_{t},x^{n}\right\rangle -{\bar{b}}%
_{t}=f_{D}\left( x^{n}\right) ,~t\in T\right\} \right. $ \\
& $\left. \cup \left\{ -a_{t}\mid -\left\langle a_{t},x^{n}\right\rangle +{%
\bar{b}}_{t}=f_{D}\left( x^{n}\right) ,~t\in D\right\} \right) ,$%
\end{tabular}%
\end{equation*}%
(where we have appealed to the Ioffe-Tikhomirov theorem --see, for instance,
{\cite[Theorem 2.4.18]{Zal02}} {or, in the finite dimensional setting, to a
consequence of the Danskin Theorem \cite{danskin}}-- and {to} the classical {%
Mazur} theorem). Therefore we can write
\begin{equation}
u^{n}=\sum_{t\in T}\lambda _{t}^{n}a_{t}+\sum_{t\in D}\mu _{t}^{n}\left(
-a_{t}\right)   \label{eq_u^r}
\end{equation}%
where, for all $n\in \mathbb{N},$ $\lambda _{t}^{n}\geq 0$ for $t\in T,$ $%
\mu _{t}^{n}\geq 0$ for $t\in D,$ $\sum_{t\in T}\lambda _{t}^{n}+\sum_{t\in
D}\mu _{t}^{n}=1,$ $\lambda _{t}^{n}>0$ for finitely many elements $t\in T$
verifying $\left\langle a_{t},x^{n}\right\rangle -{\bar{b}}_{t}=f_{D}\left(
x^{n}\right) ,$ and $-\left\langle a_{t},x^{n}\right\rangle +{\bar{b}}%
_{t}=f_{D}\left( x^{n}\right) $ whenever $\mu _{t}^{n}>0.$

Clearly, appealing to (\ref{eq_u^r}), each solution of the system
\begin{equation*}
\left\{ \left\langle a_{t},x\right\rangle \leq {\bar b}_{t},~t\in
T;~-\left\langle a_{t},x\right\rangle \leq -{\bar b}_{t},~t\in D\right\} ,
\end{equation*}%
i. e., each $x\in \mathcal{L}_{D}\left( {\bar b},-{\bar b}_{D}\right) ,$
and $\bar{x}$ in particular, satisfies the inequality
\begin{equation}
\left\langle u^n,x\right\rangle \leq \sum_{t\in T}\lambda _{t}^n{\bar b}%
_{t}+\sum_{t\in D}\mu _{t}^n\left( -{\bar b}_{t}\right) .
\label{eq_u^r consequence}
\end{equation}%
Accordingly, $\left\Vert x^n-\bar{x}\right\Vert $ is bounded from below
by the distance from $x^n$ to the solution set of (\ref{eq_u^r consequence}%
), which can be computed by means of the well-known Ascoli formula. In other
words,
\begin{equation}
\left\Vert x^n-\bar{x}\right\Vert \geq \frac{\sum_{t\in T}\lambda
_{t}^n\left( \left\langle a_{t},x^n\right\rangle -{\bar b}_{t}\right)
+\sum_{t\in D}\mu _{t}^n\left( -\left\langle a_{t},x^n\right\rangle +{\bar
b}_{t}\right) }{\left\Vert u^n\right\Vert _{\ast }}=\frac{f_{D}\left(
x^n\right) }{\left\Vert u^n\right\Vert _{\ast }},  \label{eq_x-xbar}
\end{equation}%
where the condition stated after formula (\ref{eq_u^r}) has been taken into
account.

Next we proceed with the construction of the aimed sequence $\left\{
b^{n}\right\} $ such that (\ref{eq_proof_fD}) holds. To do this, we appeal
to Urysohn's Lemma to obtain a certain $\varphi _{n}\in C\left( T,\left[ 0,1%
\right] \right) $ such that
\begin{equation*}
\varphi _{n}\left( t\right) =\left\{
\begin{tabular}{ll}
$0$ & if $t\in D,$ \\
$1$ & if $\left\langle a_{t},x^{n}\right\rangle -{\bar{b}}_{t}\leq -\left(
1+\frac{1}{n}\right) f_{D}\left( x^{n}\right) .$%
\end{tabular}%
\right.
\end{equation*}%
Certainly sets $D$ and $\left\{ t\in T\mid \left\langle
a_{t},x^{n}\right\rangle -{\bar{b}}_{t}\leq -\left( 1+\frac{1}{n}\right)
f_{D}\left( x^{n}\right) \right\} $ are closed disjoint sets in $T,$
recalling the definition of $f_{D}\left( x^{n}\right) $ and the fact that $%
f_{D}\left( x^{n}\right) >0.$

Let us define, for each $t\in T,$
\begin{equation*}
b_{t}^n=\left( 1-\varphi _n\left( t\right) \right) \left\langle
a_{t},x^n\right\rangle +\varphi _n\left( t\right) \left( {\bar b}%
_{t}+f_{D}\left( x^n\right) \right) .
\end{equation*}%
Then we can easily check that $x^n\in \mathcal{S}_{{\bar c}}\left(
b^n\right) ,$ since $x^n\in \mathcal{F}\left( b^n\right) $ and $D\subset
\mathcal{K}_{b^n}\left( x^n\right) ,$ so that the KKT conditions hold.

Finally, let us observe that $\left\vert b_{t}^n-{\bar b}_{t}\right\vert
\leq f_{D}\left( x^n\right) $ when $\varphi _n\left( t\right) =1$ and
\begin{equation*}
-\left( 1+\frac{1}n\right) f_{D}\left( x^n\right) <a_{t}^{\prime }x^n-{\
\bar b}_{t}\leq f_{D}\left( x^n\right)
\end{equation*}%
when $\varphi _n\left( t\right) <1.$ Accordingly,
\begin{equation}
\left\Vert b^n-{\bar b}\right\Vert _{\infty }\leq \left( 1+\frac{1}%
n\right) f_{D}\left( x^n\right) ,  \label{eq_br-bbar}
\end{equation}%
which, together with (\ref{eq_x-xbar}), entails (\ref{eq_proof_fD}).
\end{proof}
\begin{rem}
\emph{The role of the uniqueness assumption} $\mathcal{S}\left( {\bar{c}},{%
\bar{b}}\right) =\left\{ \bar{x}\right\} $ \emph{consists in ensuring}
$\mathcal{L}_{D}\left( {\bar{b}},-{\bar{b}}_{D}\right) \supset \mathcal{S%
}\left( {\bar{c}},{\bar{b}}\right) $ \emph{for all} $D\in \mathcal{K}_{{%
\bar{b}}}\left( \bar{x}\right) ,$ \emph{which is the fact underlying (%
\ref{eq_x-xbar}).}
\end{rem}

\section{Calmness modulus for finitely constrained programs}

In this section, we show that, in the case when $T$ is finite, the lower
bound on $\mathrm{clm}\mathcal{S}\left( \left( {\bar{c}},{\bar{b}}\right) ,%
\bar{x}\right) $ provided in the previous section is also an upper
bound, without requiring neither the Slater{\ condition} nor the uniqueness
of $\bar{x}$ as an optimal solution of $P\left( {\bar{c}},{\bar{b}}%
\right) .$

\begin{theo}
\label{Th-upperboundlim}Assume that $T$ is finite and let $\left( (\bar{%
c},{\bar{b}}),\bar{x}\right) \in \mathrm{gph}\mathcal{S}$. Then%
\begin{equation*}
\mathrm{clm}\mathcal{S}(({\bar{c}},{\bar{b}}),\bar{x})=\mathrm{clm}%
\mathcal{S}_{{\bar{c}}}({\bar{b}},\bar{x})\leq \sup_{D\in \mathcal{K}%
\left( \bar{x}\right) }\limsup_{\substack{ x\rightarrow \bar{x} \\ %
f_{D}\left( x\right) >0}}\frac{1}{d_{\ast }\left( 0_{p},\partial f_{D}\left(
x\right) \right) }.
\end{equation*}
\end{theo}
\begin{proof}
Let us write
\begin{equation}
\mathrm{clm}\mathcal{S}(({\bar{c}},{\bar{b}}),\bar{x}%
)=\lim_{n\rightarrow \infty }\frac{d\left( x^{n},\mathcal{S}({\bar{c}},{\bar{%
b}})\right) }{\left\Vert \left( c^{n},b^{n}\right) -({\bar{c}},{\bar{b}}%
)\right\Vert }  \label{2002}
\end{equation}%
with $\left( c^{n},b^{n}\right) \rightarrow ({\bar{c}},{\bar{b}})$, $\left(
c^{n},b^{n}\right) \neq (\bar{c},\bar{b})$ and $\mathcal{S}(c^{n},b^{n})\ni
x^{n}\rightarrow \bar{x}.$ Since, in ordinary linear programming,
optimality is equivalent to the KKT conditions (without requiring the Slater
{condition), for each} $n\in \mathbb{N}$, there exists $D_{n}\subset
T_{b^{n}}\left( x^{n}\right) $ with $\left\vert D_{n}\right\vert \leq p$
(because of Carath\'{e}odory's theorem) such that
\begin{equation}
-c^{n}\in \mathrm{cone}\left\{ a_{t},\text{ }t\in D_{n}\right\} .
\label{2005}
\end{equation}%
The finiteness of $T$ entails that the sequence $\left\{ D_{n}\right\}
_{n\in \mathbb{N}}$ of subsets of $T$ must contain a constant subsequence.
Without loss of generality, we may write
\begin{equation*}
D_{n}=D
\end{equation*}%
for all $n\in \mathbb{N}$. Hence, (\ref{2005}) reads as $-c^{n}\in \mathrm{%
cone}\left\{ a_{t},\text{ }t\in D\right\} $ (which is a closed cone) for all
$n\in \mathbb{N}$. Accordingly,
\begin{equation}
-{\bar{c}}\in \mathrm{cone}\left\{ a_{t},\text{ }t\in D\right\} .
\label{2010}
\end{equation}%
Moreover, $D\subset T_{b^{n}}\left( x^{n}\right) ,$ and hence, $\left\langle
a_{t},x^{n}\right\rangle =b_{t}^{n}$ for all $t\in D,$ clearly implies $%
D\subset T_{{\bar{b}}}\left( \bar{x}\right) .$ This, together with (\ref%
{2010}), entails $D\in \mathcal{K}_{{\bar{b}}}\left( \bar{x}\right) .$

In addition, recalling $D_{n}=D\subset T_{b^{n}}\left( x^{n}\right) ,$ (\ref%
{2010}) yields, for all $n,$%
\begin{equation*}
x^{n}\in \mathcal{S}_{{\bar{c}}}\left( b^{n}\right) .
\end{equation*}%
{\ Therefore (\ref{2002}) and the obvious fact that
\begin{equation*}
\left\Vert \left( c^{n},b^{n}\right) -({\bar{c}},{\bar{b}})\right\Vert \geq
\left\Vert b^{n}-{\bar{b}}\right\Vert _{\infty }
\end{equation*}%
} entail
\begin{equation*}
\mathrm{clm}\mathcal{S}(({\bar{c}},{\bar{b}}),\bar{x})\leq \underset{%
n\rightarrow \infty }{\lim \sup }\frac{d\left( x^{n},\mathcal{S}({\bar{c}},{%
\bar{b}})\right) }{\left\Vert b^{n}-{\bar{b}})\right\Vert _{\infty }}\leq
\mathrm{clm}\mathcal{S}_{{\bar{c}}}({\bar{b}},\bar{x}).
\end{equation*}%
Since, obviously from the definitions, $\mathrm{clm}\mathcal{S}_{{\bar{c}}}({%
\bar{b}},\bar{x})\leq \mathrm{clm}\mathcal{S}(({\bar{c}},{\bar{b}}),%
\bar{x}),$ the previous upper limit must be an ordinary limit and we
can write indeed
\begin{equation*}
\mathrm{clm}\mathcal{S}(({\bar{c}},{\bar{b}}),\bar{x})=\mathrm{clm}%
\mathcal{S}_{{\bar{c}}}({\bar{b}},\bar{x})=\lim_{n\rightarrow \infty }%
\frac{d\left( x^{n},\mathcal{S}({\bar{c}},{\bar{b}})\right) }{\left\Vert
b^{n}-{\bar{b}}\right\Vert _{\infty }}.
\end{equation*}

Finally, it is easy to see that $D\subset T_{b^{n}}\left( x^{n}\right) $
translates into $x^{n}\in \mathcal{L}_{D}\left( b^{n},-b_{D}^{n}\right) .$ {%
To} finish the proof, just observe that, from (\ref{2001}), (\ref{2001,5}),
and the obvious fact that $\left\Vert b^{n}-{\bar{b}}\right\Vert _{\infty
}=\left\Vert \left( b^{n},-b_{D}^{n}\right) -\left( {\bar{b}},-{\bar{b}}%
_{D}\right) \right\Vert _{\infty },$ we have \sloppy%
\begin{eqnarray*}
\lim_{n\rightarrow \infty }\frac{d\left( x^{n},\mathcal{S}({\bar{c}},{\bar{b}%
})\right) }{\left\Vert b^{n}-{\bar{b}}\right\Vert _{\infty }} &\leq &%
\underset{n\rightarrow \infty }{\lim \sup }\frac{d\left( x^{n},\mathcal{L}%
_{D}({\bar{b}},-{\bar{b}}_{D})\right) }{\left\Vert \left(
b^{n},-b_{D}^{n}\right) -\left( {\bar{b}},-{\bar{b}}_{D}\right) \right\Vert
_{\infty }} \\
&\leq &\mathrm{clm}\mathcal{L}_{D}\left( \left( {\bar{b}},-{\bar{b}}%
_{D}\right) ,\bar{x}\right) =\limsup_{\substack{ x\rightarrow \bar{%
x} \\ f_{D}\left( x\right) >0}}\frac{1}{d_{\ast }\left( {0}_{p},\partial
f_{D}\left( x\right) \right) }.
\end{eqnarray*}
\end{proof}
\begin{cor}
\label{cor1}If $T$ is finite and $\mathcal{S}({\bar{c}},{\bar{b}})=\left\{
\bar{x}\right\} ,$ then%
\begin{equation*}
\mathrm{clm}\mathcal{S}(({\bar{c}},{\bar{b}}),\bar{x})=\mathrm{clm}%
\mathcal{S}_{{\bar{c}}}({\bar{b}},\bar{x})=\sup_{D\in \mathcal{K}_{\bar{%
b}}\left( \bar{x}\right) }\limsup_{\substack{ x\rightarrow \bar{x}
\\ f_{D}\left( x\right) >0}}\frac{1}{d_{\ast }\left( {0}_{p},\partial
f_{D}\left( x\right) \right) }.
\end{equation*}
\end{cor}

\begin{proof}
The result comes straightforwardly from Theorems \ref{Th_lowerbound} and \ref%
{Th-upperboundlim}. The only detail to be taken into account is that the
Slater condition in Theorem~\ref{Th_lowerbound} is only needed at the very
beginning of the proof to guarantee the existence of $D\in \mathcal{K}_{{%
\bar{b}}}\left( \bar{x}\right) .$ When $T$ is finite, we are dealing
with ordinary linear programming problems and, as pointed out at the
beginning of the proof of Theorem \ref{Th-upperboundlim}, optimality implies
the KKT conditions without requiring the Slater{\ condition}.
\end{proof}

\section{Upper bounds on the calmness modulus}

In this section, we follow various approaches to getting upper estimates of $%
\mathrm{clm}\mathcal{S}(({\bar{c}},{\bar{b}}),\bar{x})$ for a given
$\left( ({\bar{c}},{\bar{b}}),\bar{x}\right) \in \mathrm{gph}%
\mathcal{S}.$ In {Subsection 5.1}, $\mathrm{clm}\mathcal{S}(({\bar{c}},{\
\bar{b}}),\bar{x})$ is related to $\mathrm{clm}\mathcal{L}%
((\left\langle {\bar{c}},\bar{x}\right\rangle ,{\bar{b}}),\bar{%
x})$, see {Theorem~\ref{neige}} for semi-infinite programs without requiring
the uniqueness of $\bar{x}$ as an optimal solution of $P({\bar{c}},{\
\bar{b}}).$ On the other hand, {\ Subsection 5.2} is concerned with finitely
constrained problems with unique optimal solutions. The main result of this
subsection, Theorem \ref{Th upper bound}, reminds of Theorem \ref{Th_cotainf}
(which deals with the Lipschitz modulus), although the main perturbation
ideas underlying the latter cannot be transferred to the calmness modulus,
as we try to illustrate in the next section (and particularly in Example \ref%
{Exa not attained}).

\subsection{A level set approach}

In this subsection, we consider a given $\left( ({\bar{c}},{\bar{b}}),%
\bar{x}\right) \in \mathrm{gph}\mathcal{S}$ and assume that $P({\bar{c}}%
,{\bar{b}})$ satisfies the Slater{\ condition}. Let us recall the %levet
level set mapping $\mathcal{L}$ defined in (\ref{eq_levelset}). In order to
obtain an upper bound %of
on $\mathrm{clm}\mathcal{S}(({\bar{c}},{\bar{b}}),\bar{x})$ in terms of
$\mathrm{clm}\mathcal{L}((\left\langle \bar{c},\bar{x}%
\right\rangle ,{\bar{b}}),\bar{x}),$ first we need a Gauvin-type
refinement of the KKT conditions (Proposition \ref{Prop_CLFM}). This is done
in Lemma~\ref{Lem Gauvin}. Before that, we need some additional notation.

For any $\left( b,x\right) \in \mathrm{gph}\mathcal{F},$ let
\begin{equation*}
C_{b}\left( x\right) :=\mathrm{co}\left\{ a_{t}\mid t\in T_{b}\left(
x\right) \right\} .
\end{equation*}%
The compactness of $T$ and $T_{{\bar b}}\left( x\right) $ for any $x\in
\mathcal{F}\left( {\bar b}\right) $ easily entails the following
equivalences.

\begin{rem}
The following conditions are equivalent:

$(i)$ $P(\bar{c}, \bar{b})$ satisfies the Slater\ condition;

$(ii)$ There exists $x^{0}\in \mathcal{F}\left( \bar{b}\right) $
and some associated $z\in \mathbb{R}^{p},$ $\Vert z\Vert =1$ 
satisfying   $a_{t}^{\prime }z>0$  for all  $t\in T_{{\bar{b}}}\left(
x^{0}\right) $;

$(iii)$ For each $x\in \mathcal{F}\left( {\bar{b}}\right) $,
\begin{equation}
\alpha _{{\bar{b}}}\left( x\right) :=\sup_{\left\Vert z\right\Vert
=1}\inf_{t\in T_{{\bar{b}}}\left( x\right) }a_{t}^{\prime }z>0.
\label{eq_alpha}
\end{equation}
\end{rem}
The proof of the equivalence between (i) and (ii) can be found in \cite[%
Theorem~7.2]{libro}.
Alternatively we can write%
\begin{equation*}
\alpha _{{\bar{b}}}\left( x\right) =\sup \left\{ \alpha >0\mid \left( \alpha
,z\right) \in \mathcal{G}_{{\bar{b}}}\left( x\right) \text{ for some }z\in
\mathbb{S}\right\} ,
\end{equation*}%
where $\mathbb{S}$ denotes the unit sphere in $\mathbb{R}^{p}$ and
\begin{equation*}
\mathcal{G}_{{\bar{b}}}\left( x\right) :=\left\{ \left( \alpha ,z\right) \in
\mathbb{R\times S}\mid \alpha >0\text{, }a_{t}^{\prime }z>\alpha ,\text{ }%
\forall t\in T_{{\bar{b}}}\left( x\right) \right\} .
\end{equation*}%
Let us define%
\begin{equation}
\bar{\lambda }:=\displaystyle{\min }_{\left( \alpha ,z\right) \in
\mathcal{G}_{{\bar{b}}}\left( \bar{x}\right) }\frac{-{\bar{c}}^{\prime
}z}{\alpha },  \label{eq_lambda}
\end{equation}%
which obviously satisfies
\begin{equation*}
\bar{\lambda }\leq \frac{\left\Vert {\bar{c}}\right\Vert _{\ast }}{%
\alpha _{{\bar{b}}}\left( \bar{x}\right) }.
\end{equation*}
The following lemma shows that $\bar{\lambda }$ bounds the sum of any
KKT multiplier set at $\bar{x}\in \mathcal{S}\left( {\bar c},{\bar b%
}\right) .$

\begin{lem}
\label{Lem Gauvin}Assume that $P({\bar{c}},{\bar{b}})$ satisfies the Slater{%
\ condition}, and let $\bar{x}\in \mathcal{S}\left( {\bar{c}},{\bar{b}}%
\right) .$ Then,
\begin{equation*}
\emptyset \neq \left\{ \lambda \geq 0\mid -{\bar{c}}\in \lambda C_{{\bar{b}}%
}\left( \bar{x}\right) \right\} \subset \left[ 0,\bar{\lambda }%
\right] .
\end{equation*}
\end{lem}

\begin{proof}
The KKT conditions (Proposition \ref{Prop_CLFM}) entail the non-emptiness of
the set in question. For any representation
\begin{equation*}
-{\bar{c}}=\lambda \sum\limits_{t\in T_{{ \bar{b}}}\left( \bar{x}%
\right) }\mu _{t}a_{t}\text{ }
\end{equation*}%
with $\lambda \geq 0,$ $\sum\limits_{t\in T_{{\bar{b}}}\left( \bar{x}%
\right) }\mu _{t}=1,$ $\mu _{t}$ nonnegative for all $t\in T_{{\bar{b}}%
}\left( \bar{x}\right) $ and only finitely many of them being positive,
one has, for all $\left( \alpha ,z\right) \in \mathcal{G}_{{ \bar{b}}%
}\left( \bar{x}\right) ,$
\begin{equation*}
-{\bar{c}}^{\prime }z=\lambda \sum\limits_{t\in T_{{\bar{b}}}\left(
\bar{x}\right) }\mu _{t}a_{t}^{\prime }z\geq \lambda \alpha .
\end{equation*}%
Consequently,%
\begin{equation*}
\lambda \leq \frac{-{\bar{c}}^{\prime }z}{\alpha },\text{ for all }\left(
\alpha ,z\right) \in \mathcal{G}_{{\bar{b}}}\left( \bar{x}\right) .
\end{equation*}
\end{proof}

The next lemma extends the previous one to a neighborhood of $\left( \left( {%
\bar c},{\bar b}\right) ,\bar{x}\right) $ relative to $\mathrm{gph}%
\mathcal{S}$.

\begin{lem}
\label{Lem_Gauvin2}Assume that $P({\bar{c}},{\bar{b}})$ satisfies the Slater{%
\ condition} and let $\bar{x}\in \mathcal{S}\left( {\bar{c}},{\bar{b}}%
\right) .$ Then, for any $\varepsilon >0$ there exist neighborhoods $U$ of $%
\bar{x}$ and $V$ of $\left( {\bar{c}},{\bar{b}}\right) $ such that, for
any $\left( c,b\right) \in V$ and $x\in \mathcal{S}\left( c,b\right) \cap U,$
it holds
\begin{equation}
-c\in \lambda C_{b}\left( x\right)   \label{eq_Gauvin}
\end{equation}%
for some $\lambda \in \lbrack 0,\bar{\lambda }+\varepsilon ).$
\end{lem}

\begin{proof}
First, observe that $P(c,b)$ still satisfies the Slater{\ condition} if $%
\left( c,b\right) $ is close enough to $\left( {\bar{c}},{\bar{b}}\right) ;$
so that the KKT conditions at any $\left( \left( c,b\right) ,x\right) \in
\mathrm{gph}\mathcal{S}$ with $\left( c,b\right) $ close enough to $\left(
\bar{c},{\bar{b}}\right) $ yield
\begin{equation*}
-c\in \lambda C_{b}\left( x\right) \text{ for some }\lambda \geq 0.
\end{equation*}
Reasoning by contradiction, assume the existence of $\varepsilon >0$ and a
sequence \\ $\left\{ \left( \left( c^n,b^n\right) ,x^n\right) \right\} _{n\in
\mathbb{N}}\subset \mathrm{gph}\mathcal{S}$ converging to $\left( \left( {\
\bar c},{\bar b}\right) ,\bar{x}\right) $ such that
\begin{equation*}
-c^n=\lambda ^n \sum\limits_{t\in T_{b^n}\left( x^n\right) }\mu _{t}^na_{t},
\end{equation*}%
with $\lambda ^n\geq 0,$ $%\dsum
\sum\limits_{t\in T_{b^n}\left( x^n\right) }\mu _{t}^n=1,$ $\mu _{t}^n$
nonnegative for all $t\in T_{b^n}\left( x^n\right) $ and only finitely many
of them being positive, in such a way that
\begin{equation}
\lambda ^n\geq \bar{\lambda }+\varepsilon  \label{eq_1000}
\end{equation}%
for all $n\in \mathbb{N}.$

In this case, appealing to Carath\'{e}odory's Theorem and following the
lines of \cite[Lemma 3]{CHPT12} (see also \cite[Lemma 1]{CHPT12} and
references therein), we may assume, for an appropriate subsequence of $n$'s,
\begin{equation*}
\lambda ^{n}\rightarrow \lambda ,\text{ and }\sum\limits_{t\in
T_{b^{n}}\left( x^{n}\right) }\mu _{t}^{n}a_{t}\rightarrow \sum\limits_{t\in
T_{{\bar{b}}}\left( \bar{x}\right) }\mu _{t}a_{t}
\end{equation*}%
for some $\lambda \geq 0,$ $\mu _{t}$ nonnegative for all $t\in T_{\bar{%
b}}\left( \bar{x}\right) $ and only finitely many of them being positive%
$,$ such that $%\dsum
\sum\limits_{t\in T_{{\bar{b}}}\left( \bar{x}\right) }\mu _{t}=1.$
Moreover, (\ref{eq_1000}) yields
\begin{equation*}
\lambda \geq \bar{\lambda }+\varepsilon .
\end{equation*}%
In this way, we attain %the following
a contradiction with Lemma \ref{Lem Gauvin}.
\end{proof}

The following theorem provides the aimed upper bound on $\mathrm{clm}%
\mathcal{S}(({\bar c},{\bar b}),\bar{x}).$

\begin{theo}
\label{neige} Assume that $P({\bar{c}},{\bar{b}})$ satisfies the Slater{\
condition} and let $\bar{x}\in \mathcal{S}\left( {\bar{c}},{\bar{b}}%
\right) .$ Then,%
\begin{equation*}
\mathrm{clm}\mathcal{S}(({\bar{c}},{\bar{b}}),\bar{x})\leq \max \left\{
\bar{\lambda },1\right\} \mathrm{clm}\mathcal{L}(({\bar{c}}^{\prime }%
\bar{x},{\bar{b}}),\bar{x}),
\end{equation*}%
where $\bar{\lambda }$ is defined in (\ref{eq_lambda}).
\end{theo}

\begin{proof}
If $\mathrm{clm}\mathcal{L}(({\bar c}^{\prime }\bar{x},{\bar b} ),%
\bar{x})=\infty ,$ then the inequality is trivial. Let $\mathrm{clm}
\mathcal{L}(({\bar c}^{\prime }\bar{x},{\bar b}),\bar{x}
)<\gamma <\infty $ and let $\varepsilon >0.$ Then, appealing to \cite[%
Theorem 1]{KNT10}, there exists a neighborhood $U_{1}$ of $\bar{x}$
such that \sloppy
\begin{equation}
d\left( x,\mathcal{S}({\bar c},{\bar b})\right) <\gamma \bar{f}%
\left( x\right)  \label{eq_2000}
\end{equation}
for all $x\in U_{1}$ with $\bar{f}\left( x\right) >0,$ where $\bar{%
f}$ is the supremum function defined in (\ref{eq_f_bar}). If $\left(
b,x\right) \in \mathrm{gph}\mathcal{F},$ then%
\begin{equation}
\bar{f}\left( x\right) \leq \max \{{\bar c}^{\prime }(x-\bar{x}),%
\text{ }\sup_{t\in T}(b_{t}-{\bar b}_{t})\}\leq \max \left\{ \bar{c}%
^{\prime }(x-\bar{x}),\left\Vert b-{\bar b}\right\Vert _{\infty
}\right\} .  \label{eq_3000}
\end{equation}%
By Lemma \ref{Lem_Gauvin2}, there exist neighborhoods $U_{2}$ of $\bar{x%
}$ and $V$ of $\left( {\bar c},{\bar b}\right) $ such that, for any $%
\left( c,b\right) \in V$ and $x\in \mathcal{S}(c,b)\cap U_{2},$ (\ref%
{eq_Gauvin}) holds for some $\lambda \in \lbrack 0,\bar{\lambda }%
+\varepsilon ).$ Let $\left( c,b\right) \in V$ and $x\in \mathcal{S}%
(c,b)\cap U_{1}\cap U_{2}\cap B_{\varepsilon }\left( \bar{x}\right) $
with $\bar{f}\left( x\right) >0.$ Then $-c\in %\dsum
\sum\limits_{t\in T_{0}}\lambda _{t}a_{t}$ for some finite subset $%
T_{0}\subset T_{b}\left( x\right) $ and some numbers $\lambda _{t}\geq 0,$ $%
t\in T_{0},$ satisfying $%
%\dsum
\sum\limits_{t\in T_{0}}\lambda _{t}<\bar{\lambda }+\varepsilon .$
Hence,%
\begin{eqnarray*}
-c^{\prime }(x-\bar{x}) &=& \sum\limits_{t\in T_{0}}\lambda
_{t}a_{t}^{\prime }(x-\bar{x})= \sum\limits_{t\in T_{0}}\lambda
_{t}(b_{t}-a_{t}^{\prime }\bar{x}) \\
&\geq & \sum\limits_{t\in T_{0}}\lambda _{t}(b_{t}-{\bar b}_{t})\geq -(%
\bar{\lambda }+\varepsilon )\left\Vert b-{\bar b}\right\Vert _{\infty
},
\end{eqnarray*}%
and consequently%
\begin{eqnarray}
{\bar c}^{\prime }(x-\bar{x}) &\leq &(\bar{\lambda }+\varepsilon
)\left\Vert b-{\bar b}\right\Vert _{\infty }+\left\Vert c-{\bar c}%
\right\Vert _{\ast }\left\Vert x-\bar{x}\right\Vert  \label{eq_4000} \\
&\leq &(\bar{\lambda }+2\varepsilon )\left\Vert (c,b)-({\bar c},{\
\bar b})\right\Vert .  \notag
\end{eqnarray}%
Combining (\ref{eq_2000}), (\ref{eq_3000}), and (\ref{eq_4000}), we obtain
\begin{equation*}
d\left( x,\mathcal{S}({\bar c},{\bar b})\right) \leq \gamma \max \{%
\bar{\lambda }+2\varepsilon ,1\}\left\Vert (c,b)-({\bar c},{\bar b}%
)\right\Vert .
\end{equation*}%
Hence,
\begin{equation*}
\mathrm{clm}\mathcal{S}(({\bar c},{\bar b}),\bar{x})\leq \gamma
\max \{\bar{\lambda }+2\varepsilon ,1\}.
\end{equation*}
The aimed equality follows after passing to limits in the right hand side of
the last equality as $\varepsilon \downarrow 0$ and $\gamma \downarrow
\mathrm{clm}\mathcal{L}(({\bar c}^{\prime }\bar{x},{\bar b}),%
\bar{x}).$
\end{proof}
\vskip 2mm
\textbf{5.2 An upper bound relying on the nominal data\smallskip }

Throughout this subsection we assume that:
\begin{enumerate}
\item[$(i)$] $T$ is finite;
\item[$(ii)$] ${\ }-{\bar{c}}\in $ \textrm{int}\thinspace $A_{{\bar{b}}%
}\left( \bar{x}\right) .$
\end{enumerate}
\noindent {Condition (ii) implies } $\mathcal{S}\left( {\bar{c}},{\bar{b}}%
\right) =\left\{ \bar{x}\right\} $ and $\left\vert T\right\vert \geq p$
according to a standard argument in linear programming$.$

Our goal in this subsection consists of {demonstrating that the quantity }$%
\max\limits_{D\in \mathcal{T}_{{ \bar{b}}}\left( \bar{x}\right)
}~\left\Vert A_{D}^{-1}\right\Vert ${\ constitutes an upper bound on $%
\mathrm{clm}\mathcal{S} (  ( \bar{c}, \bar{b}) ,%
\bar{x}).$ This upper bound has the virtue to be easily
computable as it relies exclusively on the nominal data $(\bar{c},
\bar{b}) $ and the nominal solution $\bar{x}.$ In {\ Section}
6, we show that this upper bound may not be attained. }
\begin{lem}
\label{Lem ratio at extreme}{\ Under conditions (i) and (ii) above, there }%
exist sequences $b^{n}\rightarrow {\bar{b}}$ and $\mathcal{S}\left( {
\bar{c}},b^{n}\right) \ni x^{n}\rightarrow \bar{x}$ with $\mathcal{T}%
_{b^{n}}\left( x^{n}\right) \neq \emptyset $ for all $n\in \mathbb{N}$ such
that 
\begin{equation*}
\mathrm{clm}\mathcal{S}_{{\bar{c}}}\left( {\bar{b}},\bar{x}\right) =%
\underset{n\rightarrow \infty }{\lim \sup }\frac{\left\Vert x^{n}-\bar{x%
}\right\Vert }{\left\Vert b^{n}-{\bar{b}}\right\Vert _{\infty }}.
\end{equation*}
\end{lem}
\begin{proof}
The only point to be proved is that sequence $\left( x^{n}\right) _{n\in 
\mathbb{N}}$ may be chosen in such a way that $\mathcal{T}_{b^{n}}\left(
x^{n}\right) \neq \emptyset $ for all $n.$ Let 
\begin{equation*}
\mathrm{clm}\mathcal{S}_{{\bar{c}}}\left( {\bar{b}},\bar{x}\right) =%
\underset{n\rightarrow \infty }{\lim \sup }\frac{\left\Vert z^{n}-\bar{x%
}\right\Vert }{\left\Vert b^{n}-{\bar{b}}\right\Vert _{\infty }}
\end{equation*}%
with $b^{n}\rightarrow {\bar{b}}$ and $\mathcal{S}\left( {\bar{c}}%
,b^{n}\right) \ni z^{n}\rightarrow \bar{x}.$ Assume that, for some $%
n\in \mathbb{N}$, $\mathcal{T}_{b^{n}}\left( z^{n}\right) =\emptyset ,$ and
let us show that $z^{n}$ may be replaced with $x^{n}\in \mathcal{S}\left( {\ 
\bar{c}},b^{n}\right) $ such that $\mathcal{T}_{b^{n}}\left( x^{n}\right)
\neq \emptyset $ and $\left\Vert x^{n}-\bar{x}\right\Vert \geq
\left\Vert z^{n}-\bar{x}\right\Vert .$ This will complete the proof.

According to the KKT conditions (recall that $T$ is finite), and taking into
account the current assumption that $\mathcal{T}_{b^{n}}\left( z^{n}\right)
=\emptyset ,$ we can write 
\begin{equation}
-{\bar{c}}=\sum_{i=1}^{k}\lambda _{i}^{n}a_{t_{i}^{n}},  \label{c_lemma}
\end{equation}%
with $k<p,$ $\lambda _{i}^{n}\geq 0$ for $1\leq i\leq k,$ $D_{n}:=\left\{
t_{1}^{n},...,t_{k}^{n}\right\} \subset T_{b^{n}}\left( z^{n}\right) $ and $%
\left\{ a_{t_{1}^{n}},...,a_{t_{k}^{n}}\right\} $ linearly independent (by
Carath\'{e}odory's Theorem). Without loss of generality, assume that $D_{n}$
is maximal, in the sense that there does not exist $\widetilde{D}%
_{n}\varsupsetneq D_{n}$ such that $\widetilde{D}_{n}\subset T_{b^{n}}\left(
z^{n}\right) ,$ $-{\bar{c}\in }\mathrm{cone}\left\{ a_{t},\text{ }t\in 
\widetilde{D}_{n}\right\} $ and $\left\{ a_{t},\text{ }t\in \widetilde{D}%
_{n}\right\} $ is linearly independent$.$ Pick $0_{p}\neq u^{n}\in \left\{
a_{t_{1}^{n}},...,a_{t_{k}^{n}}\right\} ^{\bot }$ (hence $u^{n}\bot {\bar{c}}
$). It is a basic fact that either $\left\Vert z^{n}+\alpha u^{n}-\bar{x%
}\right\Vert \geq \left\Vert z^{n}-\bar{x}\right\Vert $ for all $\alpha
>0$ or $\left\Vert z^{n}-\alpha u^{n}-\bar{x}\right\Vert \geq
\left\Vert z^{n}-\bar{x}\right\Vert $ for all $\alpha >0$ (see, \cite[
Lemma 5]{CGP08}). Without loss of
generality, assume that the first case holds. Note that $D_{n}\subset
T_{b^{n}}\left( z^{n}+\alpha u^{n}\right) $ and $z^{n}+\alpha u^{n}\in 
\mathcal{S}\left( {\bar{c}},b^{n}\right) $ whenever $z^{n}+\alpha u^{n}$ is
feasible for $b^{n}.$ On the other hand, the %boundesdess
boundedness of $\{\bar{x}\}=\mathcal{S}\left( {\bar{c}},{\bar{b}}%
\right) $ entails the boundedness of $\mathcal{S}\left( {\bar{c}}%
,b^{n}\right) $ for $n$ large enough (by the upper semicontinuity of $%
\mathcal{S}$ at $\left( {\bar{c}},{\bar{b}}\right) ;$ see, e.g., \cite[Theorem 4.3.3]{BGKKT82}). Hence, $\alpha ^{n}:=\sup \left\{ \alpha \geq
0\mid z^{n}+\alpha u^{n}\in \mathcal{F}\left( b^{n}\right) \right\} $ is
finite and, then, there exists a new index, say $t_{k+1}^{n},$ such that $%
t_{k+1}^{n}\in T_{b^{n}} \left( z^{n}+\alpha ^{n}u^{n}\right) $ and
that $\left\{ a_{t_{1}^{n}},...,a_{t_{k}^{n}},a_{t_{k+1}^{n}}\right\} $ is
linearly independent. Specifically,  the
maximality of $D_{n}$ implies that $\alpha ^{n}\neq 0\,\ $ and the definition of $\alpha ^{n}$ entails
the existence of $t_{k+1}^{n}\in T_{b^{n}}\left( z^{n}+\alpha
^{n}u^{n}\right) $ such that $\left\langle
a_{t_{k+1}^{n}},u^{n}\right\rangle >0$ yielding the linear independence of $%
\left\{ a_{t_{1}^{n}},...,a_{t_{k}^{n}},a_{t_{k+1}^{n}}\right\} .$  Obviously we can
extend (\ref{c_lemma}) by defining $\lambda _{k+1}^{n}=0.$ If $k+1=p,$ take $%
x^{n}=z^{n}+\alpha ^{n}u^{n},$ otherwise, we repeat the process until
finding the aimed point $x^{n}$ and an associated set of active indices $%
\left\{ t_{1}^{n},...,t_{p}^{n}\right\} $ belonging to $\mathcal{T}%
_{b^{n}}\left( x^{n}\right) $.

\end{proof}
\begin{theo}[Upper bound on the calmness modulus]
\label{Th upper bound}%
\begin{equation}
\mathrm{clm}\mathcal{S}\left( \left( {\bar c},{\bar b}\right) ,\bar{%
x}\right) \leq \max\limits_{D\in \mathcal{T}_{{\bar b}}\left( \bar{x}%
\right) }~\left\Vert A_{D}^{-1}\right\Vert .  \label{bounds}
\end{equation}
\end{theo}

\begin{proof}
According to the previous lemma together with the fact that $\mathrm{clm}%
\mathcal{S}(({\bar c},{\bar b}),\bar{x})=\mathrm{clm}\mathcal{S}_{{%
\bar c}}({\bar b},\bar{x}),$ established in Theorem \ref%
{Th-upperboundlim} {\ we} write \sloppy
\begin{equation*}
\mathrm{clm}\mathcal{S}\left( \left( {\bar c},{\bar b}\right) ,\bar{%
x}\right) =\lim_{n\rightarrow +\infty }\frac{\left\Vert x^n-\bar{x}%
\right\Vert }{\left\Vert b^n-{\bar b}\right\Vert _{\infty }}
\end{equation*}%
with $b^n\rightarrow {\bar b},$ $\mathcal{S}\left( {\bar c},b^n\right)
\ni x^n\rightarrow \bar{x},$ and $\mathcal{T}_{b^n}\left( x^n\right)
\neq \emptyset $ for all $n\in \mathbb{N}$. Choose $D_n\in \mathcal{T}%
_{b^n}\left( x^n\right) $ for each $n,$ and write $%
x^n=A_{D_n}^{-1}b_{D_n}^n. $ From the finiteness of $T,$ the sequence $%
\left( D_n\right) _{n\in \mathbb{N}}$ of subsets of $T$ has a constant
subsequence. Thus, we may assume (without renumbering) that $D_n=D_{0}$ for
all\thinspace $n\in \mathbb{N},$ and therefore %\begin{equation*}
$x^n=A_{D_{0}}^{-1}b_{D_{0}}^n$ %\end{equation*}%
for \thinspace $n\in \mathbb{N}$. Since $D_{0}\subset T_{b^n}\left(
x^n\right) ,$ $n=1,2,...,$ one easily checks that $D_{0}\subset T_{{\bar b}%
}\left( \bar{x}\right) ,$ and then we can write $\bar{x}%
=A_{D_{0}}^{-1}{\bar b}_{D_{0}}.$ Moreover, the fact that, for any $n, $ $%
D_{0}=D_n\in \mathcal{T}_{b^n}\left( x^n\right) $ trivially implies $%
D_{0}\in \mathcal{T}_{{\bar b}}\left( \bar{x}\right) $. Therefore,
\begin{eqnarray*}
\mathrm{clm}\mathcal{S}\left( \left( {\bar c},{\bar b}\right) ,\bar{%
x}\right) &=&\lim_{n\rightarrow +\infty }\frac{\left\Vert x^n-\bar{x}%
\right\Vert }{\left\Vert b^n-{\bar b}\right\Vert _{\infty }}\leq \underset{%
n\rightarrow + \infty }{\lim \sup }\frac{\left\Vert A_{D_{0}}^{-1}\left(
b_{D_{0}}^n-{\bar b}_{D_{0}}\right) \right\Vert }{\left\Vert b_{D_{0}}^n-{%
\bar b}_{D_{0}}\right\Vert _{\infty }} \\
&\leq &\left\Vert A_{D_{0}}^{-1}\right\Vert \leq \max\limits_{D\in \mathcal{T%
}_{{\bar b}}\left( \bar{x}\right) }~\left\Vert A_{D}^{-1}\right\Vert .
\end{eqnarray*}
\end{proof}

\section{Illustrative examples}

\begin{exa}[Upper bound in Theorem \protect\ref{Th upper bound} attained]
\label{Exa attained}\emph{Consider the parameterized linear optimization
problem, in }$\mathbb{R}^{2}$ \emph{endowed with the Euclidean norm}$:$
%$,$
\begin{equation*}
\begin{tabular}{cl}
\textrm{minimize } & $x_{1}+\frac{1}{3}x_{2}$ \\
\textrm{{subject to}$\text{ }$} & $-x_{1}\leq b_{1},$ \\
& $-x_{1}-\frac{1}{2}x_{2}\leq b_{2},$ \\
& $-x_{1}-x_{2}\leq b_{3},$%
\end{tabular}%
\end{equation*}
\end{exa}
\noindent with the nominal parameter ${\bar{b}=}${$0$}${_{3}}$ and ${\bar{c}}%
=\left( 1,1/3\right) ^{\prime }.$ In this case,
%the unique optimal solution is $\bar{x}=0_{2}.$
$\bar{x}=${$0$}${_{2}}$ is the unique optimal solution. According to
Theorem \ref{RM3}{\ (v)}, it is easy to see that $\mathcal{S}$ is Aubin
continuous (and therefore isolatedly calm) at $\left( {\bar{c}},{\bar{b}}%
\right) $ for $\bar{x}.$ Moreover,
\begin{equation*}
\mathcal{T}_{{\bar{b}}}\left( \bar{x}\right) =\left\{ \left\{
1,2\right\} ,\left\{ 1,3\right\} \right\} ,~\left\Vert A_{\left\{
1,2\right\} }^{-1}\right\Vert =\sqrt{17},~\left\Vert A_{\left\{ 1,3\right\}
}^{-1}\right\Vert =\sqrt{5}.
\end{equation*}%
(According to the last expression in (\ref{Norm calculus}), $\left\Vert
A_{D}^{-1}\right\Vert $ is the inverse of the $\left\Vert \cdot \right\Vert
_{\ast }$-distance from the origin to the boundary of the parallelogram
%hose
whose vertices are $\left\{ \pm a_{t}\mid t\in D\right\} .$) If we consider
the sequence $\left\{ \left( b^{n},x^{n}\right) \right\} _{n\in \mathbb{N}}$
in $\gph\mathcal{S}_{{\bar{c}}}$ given by $b^{n}=\left(
1/n,-1/n,0\right) ^{\prime }$ and $x^{n}=\left( -1/n,4/n\right) ,$ we have
\begin{equation}
\underset{n\rightarrow +\infty }{\lim }\frac{\left\Vert x^{n}-\bar{x}%
\right\Vert }{\left\Vert b^{n}-{\bar{b}}\right\Vert _{\infty }}=\sqrt{17}.
\label{101}
\end{equation}%
This expression, together with Theorem \ref{Th upper bound}, entails $%
\mathrm{clm}\mathcal{S}\left( \left( {\bar{c}},{\bar{b}}\right) ,\bar{x}%
\right) =\sqrt{17}.$

Observe that in this example the expression for $\mathrm{clm}\mathcal{S}%
\left( \left( {\bar{c}},{\bar{b}}\right) ,\bar{x}\right) $ provided in
Corollary \ref{cor1} can be written as
\begin{equation*}
\mathrm{clm}\mathcal{S}\left( \left( {\bar{c}},{\bar{b}}\right) ,\bar{x}%
\right) =\underset{n\rightarrow +\infty }{\lim }\frac{1}{d_{\ast }\left(
0_{2},\partial f_{D}\left( x^{n}\right) \right) }=\sqrt{17},
\end{equation*}%
where $x^{n}$ is the same as above and $D=\{1,2\}.$ The reader can easily
check that $\partial f_{D}\left( x^{n}\right) =\mathrm{co}\{a_{1},-a_{2}\}$
and $d_{\ast }\left( 0_{2},\partial f_{D}\left( x^{n}\right) \right) =1/%
\sqrt{17},$ appealing to (\ref{Norm calculus}).
\begin{rem}
If we remove the last inequality in the previous example, then
alternative  $b^{n}$ and  $x^{n}$  is given by $\left(
-1/n,1/n\right) $  and  $\left( 1/n,-4/n\right). $ Observe that
the latter point does not fulfill the referred last inequality (for index $%
t=3$). To make this point feasible, we would need to perturb ${\bar{%
b}}_{3}$ until $-3/n,$  and the limit in (\ref{101}) is  $\sqrt{%
17}/3$.  This fact is used in the following example to show that the
upper bound on the calmness modulus may not be attained.
\end{rem}
\begin{exa}[Upper bound in Theorem \protect\ref{Th upper bound} not attained]

\label{Exa not attained}\emph{Consider the linear optimization problem, in }$%
\mathbb{R}^{2}$ \emph{endowed with the Euclidean norm}$:$ %$,$
\begin{equation*}
\begin{tabular}{clc}
\textrm{minimize} & $x_{1}+\frac{1}{3}x_{2}$ &  \\
\textrm{{subject to}$\text{ }$} & $-x_{1}\leq b_{1},$ & $\left( t=1\right) $
\\
& $-x_{1}-\frac{1}{2}x_{2}\leq b_{2},$ & $\left( t=2\right) $ \\
& $-x_{1}-x_{2}\leq b_{3},$ & $\left( t=3\right) $ \\
& $-x_{1}+x_{2}\leq b_{4},$ & $\left( t=4\right) $%
\end{tabular}%
\end{equation*}
\end{exa}
\noindent with the nominal data ${\bar{b}}=0_{4},$ ${\bar{c}}=\left(
1,1/3\right) ^{\prime }$ and $\bar{x}=0_{2}.$ One has
\begin{eqnarray*}
\mathcal{T}_{{\bar{b}}}\left( \bar{x}\right)  &=&\left\{ \left\{
1,2\right\} ,\left\{ 1,3\right\} ,\left\{ 2,4\right\} ,\left\{ 3,4\right\}
\right\}  \\
\left\Vert A_{\left\{ 1,2\right\} }^{-1}\right\Vert  &=&\sqrt{17}%
,~\left\Vert A_{\left\{ 1,3\right\} }^{-1}\right\Vert =\sqrt{5},~\left\Vert
A_{\left\{ 2,4\right\} }^{-1}\right\Vert =\sqrt{17}/3,~\left\Vert A_{\left\{
3,4\right\} }^{-1}\right\Vert =1.
\end{eqnarray*}%
As in Example \ref{Exa attained}, we have Aubin continuity of $\mathcal{S}$
at these nominal data. But now, following the lines of the previous example
and remark, both $x^{n}$ and $-x^{n}$ are infeasible for $b^{n}=\left(
1/n,-1/n,0,0\right) ^{\prime },$ and we would need at least either $%
b_{4}^{n}=5/n$ or $b_{3}^{n}={\ \frac{3}{n}}$ for the feasibility of {\ $%
x^{n}$ or $-x^{n},$} respectively. Indeed, it is easy to see that for any $%
b\in \mathbb{R}^{4}$ with $\left\Vert b\right\Vert _{\infty }\leq
\varepsilon $ (for any given $\varepsilon >0$), one has $\left( \varepsilon
,0\right) ^{\prime }\in \mathcal{F}\left( b\right) $ and $\mathcal{F}\left(
b\right) \subset \mathcal{F}\left( \varepsilon e\right) ,$ with $e=\left(
1,1,1,1\right) ^{\prime };$ and consequently $\mathcal{S}\left( {\bar{c}}%
,b\right) \subset \left\{ x\in \mathcal{F}\left( \varepsilon e\right) \mid
x_{1}+\frac{1}{3}x_{2}\leq \varepsilon ,\right\} ,$ where the last
inequality is nothing else but $\left\langle {\bar{c}},x\right\rangle \leq
\left\langle {\bar{c}},\left( \varepsilon ,0\right) ^{\prime
}\right\rangle .$
We can go a bit further {\ by} saying that $x_{1}\leq
\varepsilon $ for all $x\in \mathcal{S}\left( {\bar{c}},b\right) ,$ since
otherwise neither $t=1$ nor $t=4$ might be active indices at $x$ {\ ($t=4$
cannot be active as $x_{1}>\varepsilon $ and $x_{1}+\frac{1}{3}x_{2}\leq
\varepsilon $ imply $x_{2}<0$), } and therefore $x$ would not be optimal
according to the KKT conditions. Thus,
\begin{equation*}
\mathcal{S}\left( {\bar{c}},b\right) \subset \left\{ x\in \mathcal{F}%
\left( \varepsilon e\right) \mid x_{1}+\frac{1}{3}x_{2}\leq \varepsilon
,~x_{1}\leq \varepsilon \right\} ,
\end{equation*}%
and consequently (recalling that we are working with  the
 Euclidean norm) $\left\Vert x-\bar{x}\right\Vert \leq \left\Vert 
\bar{x}-\left( \varepsilon ,-2\varepsilon \right) ^{\prime }\right\Vert
=\sqrt{5}\varepsilon $ whenever $\left\Vert b\right\Vert _{\infty }\leq
\varepsilon $ and $x\in \mathcal{S}\left( \bar{c},b\right) .$ 
Moreover $\left( \varepsilon ,-2\varepsilon \right) ^{\prime
}\in \mathcal{S}\left( {\bar{c}},b^{\varepsilon }\right) $ with $%
b^{\varepsilon }:=\left( -\varepsilon ,\varepsilon ,\varepsilon ,\varepsilon
\right) ^{\prime }$ and $\mathcal{T}_{b^{\varepsilon }}\left( \left(
\varepsilon ,-2\varepsilon \right) ^{\prime }\right) =\left\{ 1,3\right\} .$
Gathering all of this, we have \sloppy%
\begin{equation}
\mathrm{clm}\mathcal{S}\left( \left( {\bar{c}},{\bar{b}}\right) ,%
\bar{x}\right) =\underset{\varepsilon \rightarrow 0^{+}}{\lim \sup }%
\frac{\sqrt{5}\varepsilon }{\varepsilon }=\sqrt{5}.  \label{eq_exa2}
\end{equation}
Concerning the expression of Corollary \ref{cor1}, in this case we can
write, once that we know (\ref{eq_exa2})%
\begin{equation*}
\mathrm{clm}\mathcal{S}\left( \left( {\bar{c}},{\bar{b}}\right) ,%
\bar{x}\right) =\underset{\varepsilon \rightarrow 0^{+}}{\lim }\frac{1}{%
d_{\ast }\left( 0_{2},\partial f_{D}\left( x^{\varepsilon }\right) \right) }=%
\sqrt{5},
\end{equation*}%
where $x^{\varepsilon }:=\left( \varepsilon ,-2\varepsilon \right) ^{\prime }
$ and $D=\{1,3\}.$ The reader can easily check that $\partial f_{D}\left(
x^{\varepsilon }\right) =\mathrm{co}\{-a_{1},a_{3}\}$ and $d_{\ast }\left(
0_{2},\partial f_{D}\left( x^{\varepsilon }\right) \right) =1/\sqrt{5}$
appealing once more to (\ref{Norm calculus}).

The next remark emphasizes the difficulties which may arise when trying to
compute the calmness modulus in comparison with the Lipschitz modulus.

\begin{rem}
\emph{In the previous example, we may write, taking Theorem \ref{Th_cotainf}
into account,}
\begin{equation*}
\mathrm{lip}\mathcal{S}\left( \left( {\bar{c}},{\bar{b}}\right) ,\bar{x}%
\right) =\sqrt{17}=\lim_{n\rightarrow \infty }\frac{\left\Vert x^{n}-%
\widetilde{x}^{n}\right\Vert }{\left\Vert b^{n}-\widetilde{b}^{n}\right\Vert
_{\infty }}
\end{equation*}%
\emph{with} $x^{n}=\left( -1/n,4/n\right) ^{\prime },$ $\widetilde{x}^{n}=%
\bar{x}=${$0$}${_{2},}\;b^{n}=\left( 1/n,-1/n,0,5/n\right) ^{\prime },$
$\widetilde{b}^{n}=\left( 0,0,0,5/n\right) ^{\prime }.$
\end{rem}
We finish this section we a semi-infinite example which tries to show some
geometrical ideas that make us conjecture that perturbations of $\bar{c}$
are negligible when computing the calmness modulus of $\mathcal{S}$ at $%
\left( {\bar{c}},{\bar{b}}\right) ;$ i. e., $\mathrm{clm}\mathcal{S}\left(
\left( {\bar{c}},{\bar{b}}\right) ,\bar{x}\right) =\mathrm{clm}%
\mathcal{S}_{{\bar{c}}}\left( \bar{b},\bar{x}\right) .$ In this
example one has that $T$ is infinite, $P({\bar{c}},{\bar{b}})$ satisfies the
Slater{\ condition}, $\bar{x}\in \mathcal{S}({\bar{c}},{\bar{b}}),$ and
$-\bar{c}\in \mathrm{int}\,A_{\bar{b}}(\bar{x})$ (which already implies
$\mathcal{S}\left( {\bar{c}},{\bar{b}}\right) =\left\{ \bar{x}\right\} $
according to Theorem \ref{th charact_linear}).
\begin{exa}
\emph{Consider the linear optimization problem, in }$\mathbb{R}^{2}$ \emph{%
endowed with the Euclidean norm}$:$ %$,$
%\marginpar{\flushleft\small{Why 4 and 5?}}
\begin{equation*}
\begin{tabular}{clc}
\textrm{minimize} & $x_{1}$ &  \\
\textrm{{subject to}$\text{ }$} & $\left( \cos t\right) x_{1}+\left( \sin
t\right) x_{2}\leq b_{t},$ & \multicolumn{1}{l}{$t\in \left[ -\pi ,\pi %
\right] $} \\
& $-x_{1}-x_{2}\leq b_{4},$ & \multicolumn{1}{l}{$\left( t=4\right) $} \\
& $-x_{1}+x_{2}\leq b_{5}.$ & \multicolumn{1}{l}{$\left( t=5\right) $}%
\end{tabular}%
\end{equation*}
\end{exa}
\noindent According to Theorem \ref{th charact_linear}, $\mathcal{S}$ is
isolatedly calm at $\left( {\bar{c}},{\bar{b}}\right) $ for $\bar{x}%
,$ with ${\bar{c}}=\left( 1,0\right) ^{\prime },$ ${\bar{b}}_{t}=1$ for
all $t\in T:=\left[ -\pi ,\pi \right] \cup \left\{ 4,5\right\} $ and $%
\bar{x}=\left( -1,0\right) ^{\prime }.$ Here $T_{{\bar{b}}}\left(
\bar{x}\right) =\left\{ -\pi ,\pi ,4,5\right\} $ and
\begin{equation*}
\mathcal{T}_{{\bar{b}}}\left( \bar{x}\right) =\left\{ \left\{ -\pi
,4\right\} ,\left\{ -\pi ,5\right\} ,\left\{ \pi ,4\right\} ,\left\{ \pi
,5\right\} ,\left\{ 4,5\right\} \right\} .
\end{equation*}%
Note that $\left\{ -\pi ,\pi \right\} \notin \mathcal{T}_{{\bar{b}}}\left(
\bar{x}\right) $ since $a_{-\pi }$ and $a_{\pi }$ are not linearly
independent. We can easily see that $\left\Vert A_{D}^{-1}\right\Vert =\sqrt{%
5}$ for all $D\in \mathcal{T}_{{\bar{b}}}\left( \bar{x}\right)
\backslash \left\{ \left\{ 4,5\right\} \right\} $ and $\left\Vert A_{\left\{
4,5\right\} }^{-1}\right\Vert =1.$ We will show that in this example
\begin{equation*}
\mathrm{clm}\mathcal{S}\left( \left( {\bar{c}},{\bar{b}}\right) ,%
\bar{x}\right) =\mathrm{clm}\mathcal{S}_{{\bar{c}}}\left( \bar{b},%
\bar{x}\right) =\sqrt{5}.
\end{equation*}%
Let $0<\varepsilon <3-\sqrt{8}$ be given (the upper bound $3-\sqrt{8}$ will
guarantee the existence of some forthcoming elements). We can check that,
for all $\left( c,b\right) \in \mathbb{R\times }C\left( T,\mathbb{R}\right) $
with $\left\Vert \left( c,b\right) -\left( {\bar{c}},{\bar{b}}\right)
\right\Vert \leq \varepsilon ,$ we have
\begin{equation}
\mathcal{S}\left( c,b\right) \subset \left\{ x\in \mathbb{R}^{2}\left\vert
\begin{array}{l}
-x_{1}-x_{2}\leq 1+\varepsilon ,~-x_{1}+x_{2}\leq 1+\varepsilon , \\
x_{1}\leq -\sqrt{\left( 1-\varepsilon \right) ^{2}-x_{2}^{2}}%
\end{array}%
\right. \right\} .  \label{angle}
\end{equation}%
The first two inequalities need to hold for feasibility reasons. The last
inequality obeys the KKT conditions, in order to ensure the existence of
active indices at $x\in \mathcal{S}\left( c,b\right) $ to enable $-c\in
A_{b}\left( x\right) $.

Routine calculus shows that the furthest points from the origin in the
right-hand side set of (\ref{angle}) are
\begin{equation*}
\frac{1}{2}\left( -1-\varepsilon -\sqrt{1-6\varepsilon +\varepsilon ^{2}}%
,\pm \left( 1+\varepsilon -\sqrt{1-6\varepsilon +\varepsilon ^{2}}\right)
\right) \approx \left( -1+\varepsilon ,\pm 2\varepsilon \right) .
\end{equation*}
This yields $\mathrm{clm}\mathcal{S}\left( \left( {\bar{c}},{\bar{b}}%
\right) ,\bar{x}\right) \leq \sqrt{5}.$ Next we provide a sequence $%
\mathrm{gph}\mathcal{S}_{{\bar{c}}}\ni \left( b^{n},x^{n}\right)
\rightarrow \left( {\bar{b}},\bar{x}\right) $ such that
\begin{equation}
\lim_{n\rightarrow \infty }\frac{\left\Vert x^{n}-\bar{x}\right\Vert }{%
\left\Vert b^{n}-{\bar{b}}\right\Vert _{\infty }}=\sqrt{5}.
\label{ratio sqrt5}
\end{equation}%
Setting $\varepsilon =1/n$ with $n\in \mathbb{N}$ in the above expressions,
let
\begin{align*}
\widetilde{b}_{t}^{n}& :=\left\{
\begin{tabular}{cl}
$1-\dfrac{1}{n}$ & if $t\in \left[ -\pi ,\pi \right] ,\medskip $ \\
$1+\dfrac{1}{n}$ & if $t\in \left\{ 4,5\right\} ,$%
\end{tabular}%
\right.  \\
x^{n}& =\left( x_{1}^{n},x_{2}^{n}\right) :=\frac{1}{2}\left( -1-\frac{1}{n}-%
\sqrt{1-\frac{6}{n}+\frac{1}{n^{2}}},-1-\frac{1}{n}+\sqrt{1-\frac{6}{n}+%
\frac{1}{n^{2}}}\right) .
\end{align*}%
Then $T_{\widetilde{b}^{n}}\left( x^{n}\right) =\left\{ -\arccos \dfrac{%
x_{1}^{n}}{1-\frac{1}{n}},~4\right\} .$ The problem here is that making $%
x^{n}$ optimal for $\widetilde{b}^{n}$ yields a {large} (in relative terms)
perturbation of ${\bar{c}},$ namely, $c^{n}=-x^{n},$ that would enlarge
the ratio as $\displaystyle\lim_{n\rightarrow \infty }\dfrac{\left\Vert
x^{n}-\bar{x}\right\Vert }{\left\Vert (c^{n},\widetilde{b}^{n})-\left( {\
\bar{c}},{\bar{b}}\right) \right\Vert }=1.$ Thus, it would be convenient
to perform a small perturbation of $\widetilde{b}^{n}$ for getting a $b^{n}$
such that $x^{n}\in \mathcal{S}_{{\bar{c}}}\left( b^{n}\right) $ and (\ref%
{ratio sqrt5}) holds. The reader can check that such an element is given by
\begin{equation*}
b_{t}^{n}:=\left\{
\begin{tabular}{ll}
$1-\dfrac{1}{n}$ & if $\left\vert t\right\vert \leq \alpha _{n},\medskip $
\\
$\left( 1-\dfrac{1}{n}\right) \cos \left( \left\vert t\right\vert -\alpha
_{n}\right) $ & if $\left\vert t\right\vert >\alpha _{n},\medskip $ \\
$1+\dfrac{1}{n}$ & if $t\in \left\{ 4,5\right\} ,$%
\end{tabular}%
\right.
\end{equation*}%
where $\alpha _{n}:=\arccos \dfrac{x_{1}^{n}}{1-\frac{1}{n}}.$

\section{Concluding remarks}
The paper provides lower and upper estimates for the calmness modulus of $%
\mathcal{S}$ under different assumptions. The main contributions of this
work appeal to three different constants:%
\begin{align*}
C1& :=\sup_{D\in \mathcal{K}_{{\bar{b}}}\left( \bar{x}\right) }\limsup 
_{\substack{ x\rightarrow \bar{x}  \\ f_{D}\left( x\right) >0}}\frac{1}{%
d_{\ast }\left( {0}_{p},\partial f_{D}\left( x\right) \right) } \\
& =\sup_{D\in \mathcal{K}_{{\bar{b}}}\left( \bar{x}\right) }\mathrm{clm}%
\mathcal{L}_{D}(({\bar{b},-\bar{b}}_{D}),\bar{x}),\text{ \medskip } \\
C2& :=\max \left\{ \bar{\lambda },1\right\} \mathrm{clm}\mathcal{L}(({%
\bar{c}}^{\prime }\bar{x},{\bar{b}}),\bar{x}),\text{\medskip }
\end{align*}
%\\
%&&\text{and}
 \text{and}
\begin{align*}
C3&:=\max\limits_{D\in \mathcal{T}_{{\bar{b}}}\left( \bar{x}\right)
}~\left\Vert A_{D}^{-1}\right\Vert .
\end{align*}%
(Recall that set $\mathcal{K}_{{\bar{b}}}\left( \bar{x}\right) ,$
functions $f_{D},$ and mappings $\mathcal{L}_{D},$ for $D\in \mathcal{K}_{{%
\bar{b}}}\left( \bar{x}\right) ,$ are defined at the beginning of
Section 3; $\bar{\lambda }$ is defined in (30), and $\mathcal{T}_{{\ 
\bar{b}}}\left( \bar{x}\right) $ is in (15)).
The main results of the present paper concerning the semi-infinite case,
with $T$ being a compact Hausdorff space, are:
\begin{itemize}
\item  If $P\left( {\bar{c}},{\bar{b}}\right) $ satisfies  the  Slater condition,
 then:%
\begin{equation*}
\mathrm{clm}\mathcal{S}_{{\bar{c}}}\left( {\bar{b}},\bar{x}\right) \leq 
\mathrm{clm}\mathcal{S}\left( \left( {\bar{c}},{\bar{b}}\right) ,\bar{x}%
\right) \leq C2.
\end{equation*}
\item 
Moreover, if $\bar x$ is the  unique solution, then
\begin{equation*}
C1\leq \mathrm{clm}\mathcal{S}_{{\bar{c}}}\left( {\bar{b}},\bar{x}%
\right).
\end{equation*}
\end{itemize}
Observe that $C1$ and $C2$ translate the question of estimating the calmness
modulus of our optimal set mapping $\mathcal{S}$ into the --more exploited
in the literature-- question of computing the calmness moduli of appropriate
feasible set mappings. Specifically, one can find in the literature
different approaches and expressions applicable to the feasible set mapping
of semi-infinite linear inequality systems. In addition to the already
mentioned [16, Theorem 1], the reader is addressed to the basic
constraint qualification approach by Zheng and Ng \cite{ZN03}. An operative
matricial expression for the calmness modulus of feasible set mappings when
confined to finite linear inequality systems can be found in \cite{CLPT13}.
\begin{rem}
Both lower and upper bounds on $\mathrm{clm}\mathcal{S}\left( \left( {%
\bar{c}},{\bar{b}}\right) ,\bar{x}\right) $  constitute information about the
quantitative stability of the nominal problem.  An upper bound 
provides   a measure  of stability of the optimal solution  set, while  any positive lower bound constitutes a measure of its  instability.
 Specifically, in the case when $0<C1\leq \mathrm{clm}\mathcal{S}\left(
\left( {\bar{c}},{\bar{b}}\right) ,\bar{x}\right),$ given any
constant  $C$ such that  $0<C<C1\mathbf{,}$  and any $%
\varepsilon >0,$ as a direct consequence of the definitions, there
exist a point  $x$  and a parameter $\left( c,b\right) $
verifying that
\begin{equation*}
d\left( x,\mathcal{S}\left( {\bar{c}},{\bar{b}}\right) \right) >C~d\left(
\left( c,b\right) ,\left( {\bar{c}},{\bar{b}}\right) \right) ,\text{ }%
\left\Vert x-\bar{x}\right\Vert <\varepsilon \text{ \emph{and} }%
\left\Vert \left( c,b\right) -\left( \bar{c},\bar{b}\right)
\right\Vert <\varepsilon \text{ }.
\end{equation*}
\end{rem}
\begin{rem}
We point out that if we  know some upper bound for $\mathrm{clm}\mathcal{S%
}\left( \left( {\bar{c}},{\bar{b}}\right) ,\bar{x}\right) ,$ the
definition of  $C1$ sometimes provides an operative strategy for
computing   the aimed calmness modulus, provided that  $C1$ 
 is a lower bound. The strategy consists of finding  $D\in 
\mathcal{K}_{{\bar{b}}}\left( \bar{x}\right) $ and a sequence of
points  $\{x^{r}\}$ approaching $\bar{x},$ with  $%
f_{D}\left( x^{r}\right) >0$  for all $r$ and checking if 
$%
\lim_{r\rightarrow \infty }\frac{1}{d_{\ast }\left( {0}_{p},\partial
f_{D}\left( x^{r}\right) \right) }$  coincides with the known upper
bound. This, in fact, the argument followed in Example 1 for deriving the
exact value of $\mathrm{clm}\mathcal{S}\left( \left( {\bar{c}},{\bar{b}}%
\right) ,\bar{x}\right) $ In this example, $T$ is finite
and the known upper bound is  $C3$ (see the next paragraphs for
details about the finite case).
\end{rem} 
\vskip 2mm
When confined to the special case of ordinary linear  programming problems,
we have
\begin{itemize}
\item If $T$ is finite,   then 
\begin{equation*}
\mathrm{clm}\mathcal{S}_{{\bar{c}}}\left( {\bar{b}},\bar{x}\right) =%
\mathrm{clm}\mathcal{S}\left( \left( {\bar{c}},{\bar{b}}\right) ,\bar{x}%
\right) \leq C1.
\end{equation*}
\item Moreover,  if
$\mathcal{S}\left( {\bar{c}},{\bar{b}}\right) =\{%
\bar{x}\},$ then 
\begin{equation*}
%\mathrm{clm}\mathcal{S}_{{\bar{c}}}\left( {\bar{b}},\bar{x}\right) =%
\mathrm{clm}\mathcal{S}\left( \left( {\bar{c}},{\bar{b}}\right) ,\bar{x}%
\right) =C1\leq C3.
\end{equation*}
\end{itemize}

\begin{rem}
One can find in the literature many contributions to the stability
theory and sensitivity analysis in ordinary (finite) linear programming. As
a direct antecedent to the previous results, confined to the finite case, we
underline the classical works of W. Li (\cite{Li93}, \cite{Li94}). In fact,
the appearence of constant $C3$ reminds of the Lipschitz constant for 
$\mathcal{S}_{\bar{c}}$ introduced in \cite{Li93}. More in
detail, in relation to $\mathcal{S}_{\bar{c}},$ \cite{Li93}
analyzes the (global) Lipschitz property which (adapted to the current
notation) is satisfied by this mapping if there exists $\gamma >0$  such that%
\begin{equation}
d_{H}\left( \mathcal{S}_{\bar{c}}\left( b\right) ,\mathcal{S}_{%
\bar{c}}\left( b^{\prime }\right) \right) \leq \gamma d\left(
b,b^{\prime }\right) ,\text{ \emph{for all }}b\text{ \emph{and }}b^{\prime },
\label{eq_Li}
\end{equation}%
where  $d_{H}$ represents the Hausdorff distance. In contrast
with calmness property, this is not a local property in the sense that
parameters are allowed to run over the whole space; moreover, the distances
in the left hand side are between the whole optimal sets.  In any
case, it is obvious that any Lispchitz constant $\gamma $ verifying (%
\ref{eq_Li}) is, in particular, a calmness constant, and then an upper bound
on  $\mathrm{clm}\mathcal{S}_{{\bar{c}}}\left( {\bar{b}},\bar{x}\right)
.$ For comparative purposes with our Theorem 13, we recall here the
expression of $\gamma $ in \cite{Li93} under the assumption $%
\mathcal{S}\left( {\bar{c}},{\bar{b}}\right) =\{\bar{x}\}$ 
(yielding that the rank of the coefficient vectors in the left hand side of
the constraints is $p):$%
\begin{equation*}
\gamma =\max\limits_{\substack{ D\subset T,\text{ }\left\vert D\right\vert =p
\\ A_{D}~\text{is non-singular}}}~\left\Vert A_{D}^{-1}\right\Vert \text{.}
\end{equation*}%
 One can easily find examples in which 
\begin{equation*}
\emph{C3}<\gamma .
\end{equation*}%
 (See for instance [5, Example 4] where $ (C3$ provides the
exact Lispchitz modulus associated to the Aubin property.) In fact,  it is proven in \cite%
{Li93}  that  $\gamma $ is a Lispchitz constant for $%
\mathcal{S}_{{\bar{c}}}\left( {\bar{b}},\bar{x}\right) $, for all
$\bar{c}$ (observe that $\gamma $  does not depend on $%
\bar{c})$  and \cite[Theorem. 3.5]{Li93} establishes the sharpness of %
$\gamma $ for some $\bar{c}.$
\end{rem}

Finally, we underline an specific fruitful argument in the finite case (for
analyzing $\mathrm{clm}\mathcal{S}\left( \left( {\bar{c}},{\bar{b}}\right) ,%
\bar{x}\right) )$ which is no longer applicable when $T$ in infinite.
The argument reads as follows: if $T$ is finite, when we consider a sequence
of points $\{x^{r}\}$ converging to $\bar{x}$ and a sequence of
parameters $\{\left( c^{r},b^{r}\right) \}$ converging to $\left( \bar{c%
},\bar{b}\right) $ such that $x^{r}\in \mathcal{S}\left(
c^{r},b^{r}\right) $ (i.e., $x^{r}$ is a KKT point of $P\left(
c^{r},b^{r}\right) ,$ entailing 
\begin{equation*}
-c^{r}\in cone\{a_{t}:t\in D_{r}\},\text{ with }D_{r}\subset T_{b^{r}}\left(
x^{r}\right) )
\end{equation*}%
we always can extract a constant subsequence of $\left\{ D_{r}\right\} ;$
moreover, for $r$ large enough 
\begin{equation*}
T_{b^{r}}\left( x^{r}\right) \subset T_{\bar{b}}\left( \bar{x}%
\right) .
\end{equation*}%
Example 3 provides a situation in which $T$ is infinite and the previous
argument  is not applicable, and 
illustrates  the difficulty of the generalization of some results from finite
to semi-infinite programs. In fact, the problem of providing an exact
formula for $\mathrm{clm}\mathcal{S}\left( \left( {\bar{c}},{\bar{b}}\right)
,\bar{x}\right) ,$ perhaps under appropriate assumptions, still remains
open  and  constitutes one of the further lines of research. 
\vskip 2mm
\textbf{Acknowledgement.}
The authors would like to thank  the reviewers and the Associate Editor for the criticism and suggestions. The second and the last  author would like to thank  the universities of Alicante and Miguel Hern\'andez de Elche for hopsitality and support.
\bibliographystyle{siam}
\bibliography{myrefs+}
\end{document}